\documentclass[a4paper]{amsart}
\usepackage[margin=3.5cm]{geometry}
\usepackage{amsmath,amssymb}
\usepackage{hyperref}

\DeclareMathOperator{\Ad}{Ad}
\DeclareMathOperator{\Lie}{Lie}
\DeclareMathOperator{\Aut}{Aut}
\DeclareMathOperator{\Inn}{Inn}
\DeclareMathOperator{\GL}{GL}
\DeclareMathOperator{\SL}{SL}
\DeclareMathOperator{\SO}{SO}
\DeclareMathOperator{\PGL}{PGL}
\DeclareMathOperator{\Mat}{Mat}

\DeclareMathOperator{\Char}{char}

\newcommand{\deq}{\overset{\textup{def}}{=}}
\newcommand{\ssep}{\,:\,}
\newcommand{\Kder}{[K^{\circ},K^{\circ}]}

\newtheorem{theorem}{Theorem}[section]
\newtheorem{lemma}[theorem]{Lemma}
\newtheorem{corollary}[theorem]{Corollary}
\newtheorem{proposition}[theorem]{Proposition}

\newtheorem{maintheorem}{Theorem}
\newtheorem{maincorollary}[maintheorem]{Corollary}

\theoremstyle{definition}
\newtheorem{remark}[theorem]{Remark}

\newtheorem{definition}[theorem]{Definition}
\newtheorem{example}[theorem]{Example}

\newtheorem*{question*}{Question}
\newtheorem*{remark*}{Remark}

\title[Relative c.\ r.\ and normalised subgroups]{Relative complete reducibility and normalised subgroups}
\author{Maike Gruchot, Alastair Litterick, and Gerhard R\"{o}hrle}

\address{Maike Gruchot: Fakult\"{a}t f\"{u}r Mathematik, Ruhr-Universit\"{a}t Bochum, Universit\"{a}tsstra{\ss}e 150, D-44780 Bochum, Germany}
\email{maike.gruchot@rub.de}
\curraddr 
{Lehrstuhl f\"{u}r Algebra und Zahlentheorie, RWTH Aachen University, 
	D-52062 Aachen, Germany}
\email{maike.gruchot@rwth-aachen.de}

\address{Alastair Litterick: Fakult\"{a}t f\"{u}r Mathematik, Ruhr-Universit\"{a}t Bochum, Universit\"{a}tsstra{\ss}e 150, D-44780 Bochum, Germany \and Fakult\"{a}t f\"{u}r Mathematik, Universit\"{a}t Bielefeld, Postfach 100131, D-33501 Bielefeld, Germany}
\email{ajlitterick@gmail.com}
\curraddr
{Department of Mathematical Sciences, University of Essex, Wivenhoe Park, Colchester, Essex CO4 3SQ, United Kingdom}
\email{ajlitterick@essex.ac.uk}

\address{Gerhard R\"{o}hrle: Fakult\"{a}t f\"{u}r Mathematik, Ruhr-Universit\"{a}t Bochum, Universit\"{a}tsstra{\ss}e 150, D-44780 Bochum, Germany}
\email{gerhard.roehrle@rub.de}


\begin{document}

\begin{abstract}
We study a relative variant of Serre's notion of $G$-complete reducibility for a reductive algebraic group $G$. We let $K$ be a reductive subgroup of $G$, and consider subgroups of $G$ which normalise the identity component $K^{\circ}$. We show that such a subgroup is relatively $G$-completely reducible with respect to $K$ if and only if its image in the automorphism group of $K^{\circ}$ is completely reducible. This allows us to generalise a number of fundamental results from the absolute to the relative setting. We also derive analogous results for Lie subalgebras of the Lie algebra of $G$, as well as `rational' versions over non-algebraically closed fields.
\end{abstract}
\maketitle
\tableofcontents

\section{Introduction}

Let $G$ be a (possibly disconnected) reductive algebraic group over an algebraically closed field. In recent years much effort has been devoted to understanding Serre's notion of $G$-complete reducibility \cite{Bate2018,Bate2005,Bate2008,Bate2010,Bate2011,Liebeck1996,Litterick2018,Serre1998,Serre2003-2004,Stewart2013}. This powerful notion links the subgroup structure of $G$ with representation theory (which can be viewed as the special case $G = \GL_{n}$) and with concepts from geometric invariant theory.

In \cite{Bate2011} a relative version of this concept is introduced. If $K$ is a closed reductive subgroup of $G$, then a closed subgroup $H$ of $G$ is called \emph{relatively $G$-completely reducible with respect to $K$} if, whenever $H$ is contained in an R-parabolic subgroup $P_{\lambda}$ for a cocharacter $\lambda$ of $K$, then $H$ is contained in an R-Levi subgroup $L_{\mu}$ for some cocharacter $\mu$ of $K$ with $P_{\lambda} = P_{\mu}$. More detailed definitions are given in Section~\ref{sec:notation}.

Our main result is a direct relation between this relative variant and the usual `absolute' notion of complete reducibility (which is the case $K = G$) when considering subgroups which normalise the identity component $K^{\circ}$ of $K$. This allows us to deduce relative versions of many pivotal theorems directly from their absolute counterparts, which we do in Section~\ref{sec:corollaries}. Throughout, all groups are defined over a fixed field $k$ which, until Section~\ref{sec:rationality}, is taken to be algebraically closed.

\begin{maintheorem} \label{THM:MAIN}
Let $K \le G$ be reductive algebraic groups, write $N = N_{G}(K^{\circ})$, $C = C_{G}(K^{\circ})$, and let $\pi \colon N \to N/C$ be the quotient map. 
Let $H$ be a closed subgroup of $N$. Then $H$ is relatively $G$-completely reducible with respect to $K$ if and only if $\pi(H)$ is $\pi(N)$-completely reducible.
\end{maintheorem}
Note that we do not assume that any of the groups $G$, $H$, $K$, $N$ or $C$ in Theorem~\ref{THM:MAIN} is connected. Since $K^{\circ}$ is reductive, $\pi(N)$ is a finite extension of the connected semisimple group $\Inn(K^{\circ}) \cong K^{\circ}/Z(K^{\circ})$, hence it is reductive and $\pi(N)$-complete reducibility makes sense, cf.\ \cite[\S 6]{Bate2005}.

The assumption that $H$ normalises $K^{\circ}$ is a natural one. It holds in the absolute case or if $H$ is contained in $K$, and is often necessary to obtain sensible statements for relative results. More importantly, many natural extensions of results fail without this assumption, as we now illustrate. Recall that a $G$-completely reducible subgroup is necessarily reductive \cite[Property 4]{Serre1998}, and the converse to this holds in characteristic zero \cite[Th\'{e}or\`{e}me~4.4]{Serre2003-2004}. However neither direction holds in full generality in the relative setting, as pointed out in \cite[Remarks 3.2]{Bate2011}. Nevertheless, the following consequence of Theorem~\ref{THM:MAIN} faithfully generalises this to the relative setting, as long as $H$ normalises $K^{\circ}$. We discuss this further in Section~\ref{sec:corollaries}.

\begin{maincorollary} \label{cor:NORMAL}
Let $K \le G$ be reductive algebraic groups and let $H \le N_{G}(K^{\circ})$. If $H$ is relatively $G$-completely reducible with respect to $K$ then the unipotent radical $R_{u}(H)$ centralises $K^{\circ}$. Conversely, if $k$ has characteristic zero and $R_{u}(H)$ centralises $K^{\circ}$ then $H$ is relatively $G$-completely reducible with respect to $K$.
\end{maincorollary}
As in the absolute case, the reverse direction of Corollary~\ref{cor:NORMAL} also holds if the characteristic $\Char(k)$ of $k$ is sufficiently large relative to $K$ and $H$ (independently of $G$), and this specialises to \cite[Th\'{e}or\`{e}me~4.4]{Serre2003-2004} in the case $K = G$. A more detailed statement is given in Theorem~\ref{thm:bigp}.

Just as $G$-complete reducibility can be expressed in terms of the closure of orbits of $G$ on its Cartesian products $G^{n}$ under conjugation \cite[Corollary 3.7]{Bate2005}, the relative notion can be characterised in terms of the closure of orbits of $K$ on $G^{n}$ (see Theorem~\ref{thm:gcr_gen_tup} below). Since $\pi(N)^{\circ} = \pi(K^{\circ})$ and an orbit $K \cdot \mathbf{h}$ is closed in  $G^{n}$ if and only if $K^{\circ} \cdot \mathbf{h}$ is closed in  $G^{n}$, Theorem~\ref{THM:MAIN} is equivalent to the following.

\begin{maintheorem} \label{THM:GEOMETRIC}
Let $K \le G$ be reductive algebraic groups, write $N = N_{G}(K^{\circ})$, $C = C_{G}(K^{\circ})$, and let $\pi \colon N \to N/C$ be the quotient map. Let $\mathbf{h} \in N^{n}$ $(n \ge 1)$, and write $\pi$ also for the induced map $N^{n} \to (N/C)^{n}$. Then $K \cdot \mathbf{h}$ is closed in $G^{n}$ if and only if $\pi(N) \cdot \pi(\mathbf{h})$ is closed in $(N/C)^{n}$.
\end{maintheorem}

In the absolute case, $G$-complete reducibility of a subgroup $H$ interacts with separability of $H$ in $G$ \cite[Definition~3.27]{Bate2005} and semisimplicity of the Lie algebra $\Lie(G)$ as an $H$-module \cite{Bate2010,Uchiyama2017}. In Section \ref{sec:separabilty} we define a relative version of separability, and show that the corresponding results still hold. The next two results are particularly interesting.

The following is a generalisation of \cite[Theorem 1.7]{Bate2010} both to the relative setting and to the case that $G$ may not be connected. Recall that a prime $p$ is called \emph{bad} for the reductive group $G$ if $p$ divides some coefficient when the highest root in the root system of some simple factor of $G$ is expressed as a sum of simple roots. The characteristic $\Char(k)$ of $k$ is called \emph{good} for $G$ if it is not a bad prime. If $\Char(k)$ is good and not a divisor of $r+1$ whenever $G$ has a simple factor $A_r$ then it is called \emph{very good} for $G$. 

\begin{maintheorem} \label{THM:separable}
Let $K \le G$ be reductive algebraic groups, write $N = N_{G}(K^{\circ})$, $C = C_{G}(K^{\circ})$, and let $\pi \colon N \to N/C$ be the quotient map. Suppose that $\Char(k)$ is zero or is very good for $K$ and does not divide $|\pi(N)/\pi(N^{\circ})|$. If $H \le N$ and $\Lie(K)$ is semisimple as an $H$-module then $H$ is relatively $G$-completely reducible with respect to $K$.
\end{maintheorem}
Theorem~\ref{THM:separable} in fact holds under a slightly weaker condition on $\Char(k)$ involving $H$ as well as $K$ and $N_{G}(K^{\circ})$, see Corollary~\ref{cor:vg_sep}.

Next, for subgroups $H$ and $K$ of an algebraic group $G$ we say that $(G,K)$ is a \emph{reductive pair for $H$} if $\Lie(K)$ is an $H$-module direct summand of $\Lie(G)$ (Definition \ref{def:reductivepair}). The interplay between this notion and separability (Lemma~\ref{lem:sep_red_sep}) gives a further condition for $G$-complete reducibility (Corollary~\ref{cor:sep_red_sep}). Due to the well-known fact that every subgroup of $\GL(V)$ is separable, we obtain a particularly nice criterion in this case, with no dependence on $\Char(k)$. This is orthogonal to other criteria derived recently in \cite{Bate2018}. The proof can be found at the end of Section~\ref{sec:separabilty}.
\begin{maintheorem} \label{thm:GLV}
Let $G = \GL(V)$, let $K \le G$ be reductive and let $H \le N_{G}(K^{\circ})$. Suppose that $(G,K)$ is a reductive pair for $H$. If $\Lie(K)$ is semisimple as an $H$-module then $H$ is relatively $G$-completely reducible with respect to $K$.
\end{maintheorem}

For general connected $G$ we obtain the following criterion on $\Char(k)$ for complete reducibility, which depends on $G$ but not on $K$, $N_{G}(K^{\circ})$ or $H$.
\begin{maintheorem} \label{thm:G-connected-red-pair}
Maintain the notation of Theorem~\ref{THM:MAIN} and suppose that $G$ is connected. Suppose that $(G,K)$ is a reductive pair for $H$ and that $\Char(k)$ is very good for $G$. If $\Lie(K)$ is semisimple as an $H$-module, then $H$ is relatively $G$-completely reducible with respect to $K$.
\end{maintheorem}

The paper is organised as follows. After recalling relevant background in Section~\ref{sec:notation}, we prove Theorem~\ref{THM:MAIN} in Section~\ref{sec:proof}. We also briefly discuss the analogue of Theorem~\ref{THM:MAIN} for relative $G$-irreducibility (Corollary \ref{cor:irred}). In Section~\ref{sec:corollaries} we derive a series of consequences of Theorem~\ref{THM:MAIN}, including Corollary~\ref{cor:NORMAL}. In Section~\ref{sec:separabilty} we introduce our relative notions of separability and reductive pairs, and generalise further results from the absolute case, particularly from \cite{Bate2010}. In Section \ref{sec:liealg} we consider relative $G$-complete reducibility of Lie subalgebras of $\Lie(G)$, and derive variants of Theorems~\ref{THM:MAIN} and \ref{THM:GEOMETRIC}. Finally, in Section \ref{sec:rationality} we consider a rational version of relative $G$-complete reducibility, dropping the assumption that $k$ is algebraically closed. Again, variants of Theorems~\ref{THM:MAIN} and \ref{THM:GEOMETRIC} hold in this setting.

\section{Notation and preliminaries} \label{sec:notation}

Throughout, $k$ denotes a field and $\Char(k)$ denotes the characteristic of $k$. We take $k$ to be algebraically closed, until Section~\ref{sec:rationality} where we generalise our main results. All groups encountered are affine algebraic groups, meaning Zariski-closed subgroups of some general linear group over $k$ (or its algebraic closure $\overline{k}$ in Section~\ref{sec:rationality}). Homomorphisms between groups are morphisms of varieties, and subgroups are closed. For an algebraic group $G$, we denote by $G^{\circ}$ the connected component of the identity element. For a subgroup $K$ of $G$, the normaliser and centraliser of $K$ in $G$ are respectively denoted $N_{G}(K)$ and $C_{G}(K)$. The unipotent radical of $G$, denoted $R_{u}(G)$, is the (unique) maximal connected normal unipotent subgroup of $G$. We say that $G$ is reductive if $R_{u}(G) = 1$. We do not require that a reductive group is connected. An isogeny means a surjective map $G_{1} \to G_{2}$ with finite kernel, where $G_{1}$ and $G_{2}$ are reductive. In this case the kernel centralises $G_{1}^{\circ}$.

A cocharacter of $G$ is a homomorphism of algebraic groups from the multiplicative group $k^{\ast}$ to $G$. The set of all cocharacters of $G$ is denoted $Y(G)$. Conjugation induces an action of $G$ on $Y(G)$, with $(g \cdot \phi)(c) \deq g \phi(c) g^{-1}$ for all $g \in G$, $\phi \in Y(G)$, $c \in k^{\ast}$. We also use a dot to denote left conjugation of $G$ on itself and on the Cartesian products $G^{n}$. If $\phi \colon k^{\ast} \to G$ is a morphism, we say that the limit $\lim_{a \to 0} \phi(a)$ exists if there is a morphism $\widehat{\phi} \colon k \to G$ whose restriction to $k^{\ast}$ is $\phi$. In this case, we write $\lim_{a \to 0} \phi(a)$ for $\widehat{\phi}(0)$. If the limit exists then it is unique, as $k^{\ast}$ is Zariski-dense in $k$. For $\lambda \in Y(G)$, define the following subgroups of $G$:
\begin{align*}
P_{\lambda} &\deq \left\{ g \in G \ssep \lim_{a \to 0} (\lambda(a) \cdot g) \textup{ exists} \right\},\\
L_{\lambda} &\deq \left\{ g \in G \ssep \lim_{a \to 0} (\lambda(a) \cdot g) = g \right\} = C_{G}(\lambda(k^{\ast})).
\end{align*}
Following \cite{Bate2005}, such a subgroup $P_{\lambda}$ is called a Richardson parabolic subgroup, or R-parabolic subgroup of $G$, and $L_{\lambda}$ is called a Richardson Levi subgroup of $G$, or R-Levi subgroup. If $G$ is connected reductive then these definitions coincide with the usual definitions of parabolic subgroups and Levi subgroups \cite[\S 6]{Bate2005}. The unipotent radical of $P_{\lambda}$ is given by
\[ R_{u}(P_{\lambda}) = \left\{ g \in G \ssep \lim_{a \to 0} (\lambda(a) \cdot g) = 1 \right\}. \]

If $K$ is a reductive subgroup of $G$ then the inclusion $K \to G$ induces an injective map $Y(K) \to Y(G)$, and we identify $Y(K)$ with its image in $Y(G)$. The following is the central notion of the paper.
\begin{definition} \label{def:gcr}
Let $H$ and $K$ be closed subgroups of an algebraic group $G$. We say that \emph{$H$ is relatively $G$-completely reducible with respect to $K$} if, for every cocharacter $\lambda \in Y(K)$ such that $H \le P_{\lambda}$, there exists a cocharacter $\mu \in Y(K)$ such that $P_{\lambda} = P_{\mu}$ and $H \le L_{\mu}$.
\end{definition}

While the above definitions make sense without assumptions on $G$ or $K$, we always assume that both $K$ and $G$ are reductive. We refer to the case $K = G$ as the \emph{absolute case}, and say that $H$ is \emph{$G$-completely reducible} if the above holds in this case. For brevity we sometimes write `relatively $G$-cr with respect to $K$', or just `$G$-cr' in the absolute case.

Recall from \cite[Definition 5.4]{Bate2013} that a \emph{generic tuple} for a subgroup $H \le G$ is an $n$-tuple $\mathbf{h} \in G^{n}$ $(n \ge 1)$ such that the elements in $\mathbf{h}$ generate the same associative $k$-subalgebra of $\Mat_{m \times m}(k)$ as $H$, for some faithful representation $G \to \GL_{m}(k)$. For instance a tuple of elements which topologically generate $H$ is a generic tuple for $H$ \cite[Remark 5.6]{Bate2013}.

\begin{theorem}[{\cite[Theorem 1.1]{Bate2011}}] \label{thm:gcr_gen_tup}
Let $K \le G$ be reductive algebraic groups, let $H$ be a closed subgroup of $G$ and let $\mathbf{h} \in G^{n}$ be a generic tuple for $H$. Then $H$ is relatively $G$-completely reducible with respect to $K$ if and only if the orbit $K\cdot \mathbf{h}$ is closed in $G^{n}$.
\end{theorem}

\begin{remark}
Let $H$ be a subgroup of $G$ and let $\mathbf{h} \in G^{n}$ be a generic tuple for $H$. Then we can lengthen $\mathbf{h}$ so that $\pi(\mathbf{h})$ is also a generic tuple for $\pi(H)$. Together with Theorem~\ref{thm:gcr_gen_tup}, this shows that Theorem~\ref{THM:GEOMETRIC} implies Theorem~\ref{THM:MAIN}. Conversely, given $\mathbf{h} \in G^{n}$ we let $H$ be the closed subgroup topologically generated by the elements of $\mathbf{h}$. Then $\pi(H)$ is topologically generated by $\pi(\mathbf{h})$ and it follows at once that Theorem~\ref{THM:MAIN} implies Theorem~\ref{THM:GEOMETRIC}.
\end{remark}

\begin{remark} \label{rem:GcrMcr} Although our methods make intrinsic use of the fact that $K$ is reductive, Theorem~\ref{thm:gcr_gen_tup} shows that the assumption that $G$ is reductive is unimportant in the relative setting, since we are free to replace $G$ with any closed subgroup of $G$ containing $H$ and $K$, or with any group $G'$ containing $G$ as a closed subgroup, cf.\ \cite[Corollary 3.6]{Bate2011}. Although it is beyond the scope of this paper, there is the potential to develop the theory of relative complete reducibility with respect to arbitrary closed subgroups $K$. For instance, applying Definition~\ref{def:gcr} in the case when $Y(K)$ is trivial, i.e.\ when $K^{\circ}$ is unipotent, we see that every subgroup is relatively $G$-cr with respect to $K$. This corresponds to the geometric fact that all orbits of a unipotent group on a variety are closed.
\end{remark}

\section{Proof of Theorem \ref{THM:MAIN}} \label{sec:proof}

Recall our set-up that $G$ is a (not necessarily connected) reductive algebraic group, and $H$ and $K$ are closed subgroups of $G$ with $K$ reductive. Write $N = N_{G}(K^{\circ})$, $C = C_{G}(K^{\circ})$, let $\pi \colon N \to N/C$ be the quotient map, and assume that $H \le N$. We begin with a crucial lemma. This is straightforward if $N$ is reductive, cf.\ \cite[Lemma 6.14]{Bate2005}; but see Remark~\ref{rem:nonsplit} below for a subtle point which arises when $N$ has a non-trivial unipotent radical.

\begin{lemma} \label{lem:usefulsub} In the setting of Theorem~\ref{THM:MAIN}, there exists a reductive subgroup $M$ of $N$ such that 
\begin{enumerate}
	\item $M^{\circ} = \Kder$, $M \cap C$ is finite and $N = MC$; \label{usefulsub-i}
	\item the restriction of $\pi$ to $M$ gives an isogeny $M \to N/C$; \label{usefulsub-ii}
	\item for all $\lambda \in Y(K)$ we have $P_{\lambda} \cap N = (P_{\lambda} \cap M)C$ and $L_{\lambda} \cap N = (L_{\lambda} \cap M)C$. \label{usefulsub-iii}
\end{enumerate}
\end{lemma}

\begin{proof}Note firstly that $N^{\circ} = \Kder C^{\circ}$ since $N/C$ is a finite extension of $\Inn(K^{\circ}) = K^{\circ}/Z(K^{\circ}) = \Kder/Z(\Kder)$. 
Consider the quotient map $\sigma \colon N \twoheadrightarrow N/\Kder$. By the above, $N/\Kder$ is a finite extension of $\sigma(\Kder C)$. 
By \cite[Theorem 1.1]{Brion2015}, $N/\Kder$ admits a finite subgroup $F$ such that $N/\Kder = F(\sigma(\Kder C))$. Let $M$ be the pre-image of $F$ under 
$\sigma$. Then $M^{\circ} = \Kder$ by construction, and $M \cap C$ is finite since it is a finite extension of $\Kder \cap C$, which is itself finite as it is contained in $Z(\Kder)$. Finally, $N = \sigma^{-1}(F \sigma(\Kder C)) = M \sigma^{-1}(\sigma(C)) \le MC$, hence $N = MC$. Thus \ref{usefulsub-i} is proved. Now \ref{usefulsub-ii} and \ref{usefulsub-iii} follow directly from the fact that $N = MC$, $M \cap C$ is finite and $C \le L_{\lambda} \le P_{\lambda}$ for all $\lambda \in Y(K)$.
\end{proof}

\begin{remark} \label{rem:nonsplit}
If $N$ is reductive then \cite[Lemma 6.14]{Bate2005}, lets us take $M^{\circ}$ to be the product of the simple factors of $N^{\circ}$ which commute with $C^{\circ}$ (which is also reductive in this case). The following example shows that in general $M$ need not be isomorphic to its image in the reductive group $N/R_{u}(C)$; in particular $M$ could be a non-split extension of $M \cap R_{u}(C)$.

Let $K = \SL_{3}(k)$ where $\Char(k) = 2$ and let $C$ be a copy of $\mathbb G_a$. Let $X = \left<x\right>$ be cyclic of order $4$, and consider the product $(X \times C)/ \left<x^{2}y\right>$ for any non-zero $y \in k$. Write $H$ for the resulting abelian two-fold non-split extension of $C$, so $H^{\circ} = C$.

Now let $H$ act on $K$, such that $C$ acts trivially and $H/C$ acts as the inverse-transpose automorphism, and write $G = N$ for the semidirect product $K \rtimes H$. By construction, $K = K^{\circ}$ is normal in $N$, and $C = H^{\circ} = C_{N}(K^{\circ})$ is connected and unipotent. In $N/C$, the subgroup $M$ guaranteed by \cite[Lemma 6.14]{Bate2005} is the image of the outer automorphism, i.e.\ there is an involution giving a complement to $K = (N/C)^{\circ}$. But in $N$ this element has order $4$ by construction. So the smallest $M$ satisfying the conclusion of Lemma~\ref{lem:usefulsub} is a $4$-fold extension $\SL_{3}(k) \rtimes \left<x\right>$, where $x$ acts as the inverse-transpose automorphism on $\SL_{3}(k)$; this is a non-split extension of $M \cap C = \left<x^2\right>$.
\end{remark}

\begin{proof}[Proof of Theorem \ref{THM:MAIN}] Let $G$, $K$, $H$, $N = N_{G}(K^{\circ})$, $C = C_{G}(K^{\circ})$ and $\pi \colon N \to N/C$ be as in Theorem~\ref{THM:MAIN}, and let $M$ be the reductive subgroup given by Lemma~\ref{lem:usefulsub}.

To begin, for $\lambda \in Y(K)$ we have $C \le L_{\lambda} \le P_{\lambda}$. Hence $HC$ is contained in $P_{\lambda}$ (respectively $L_{\lambda}$) if and only if $H$ is contained in $P_{\lambda}$ (respectively $L_{\lambda}$). Moreover, since $N = MC$ we have $HC = (HC \cap M)C$, and $HC$ is contained in $P_{\lambda}$ (respectively $L_{\lambda}$) if and only if $HC \cap M$ is contained in $P_{\lambda}$ (respectively $L_{\lambda}$). Since also $\pi(H) = \pi(HC) = \pi(HC \cap M)$ it suffices to prove the conclusion of Theorem~\ref{THM:MAIN} for subgroups of the form $HC \cap M$, in particular for subgroups of $M$.

So assume $H \le M$. Then the reductive subgroup $MZ(K^{\circ})^{\circ}$ contains both $K^{\circ}$ and $H$. Since also $Y(K) = Y(K^{\circ})$, as in Remark~\ref{rem:GcrMcr} it follows that $H$ is relatively $G$-cr with respect to $K$ if and only if $H$ is relatively $MZ(K^{\circ})^{\circ}$-cr with respect to $K^{\circ}$. Moreover, because $\pi(MZ(K^{\circ})^{\circ}) = \pi(M) = \pi(N)$, it suffices to prove the conclusion of Theorem~\ref{THM:MAIN} when $G = N = MZ(K^{\circ})^{\circ}$, a reductive group in which $C^{\circ} = C_{G}(K^{\circ})^{\circ}$ is a central torus of $G^{\circ}$. We can still assume $H \le M$.

Now, under these assumptions we have $Y(G) = Y(K)$, so $H$ is relatively $G$-cr with respect to $K$ if and only if $H$ is $G$-cr. In this situation the statement of Theorem~\ref{THM:MAIN} becomes: ``Let $G$ be reductive and let $K$ be a reductive subgroup of $G$ such that $K^{\circ}$ is normal in $G$. Let $\pi \colon G \to G/C_{G}(K^{\circ})$ where $C_{G}(K^{\circ})^{\circ}$ is a central torus of $G^{\circ}$. Then $H$ is $G$-cr if and only if $\pi(H)$ is $(G/C_{G}(K^{\circ}))$-cr.'' This is proved in \cite[\S 6.2]{Bate2005}. More specifically, the result for connected groups is part of \cite[Lemma 2.12(ii)(b)]{Bate2005}, and \cite[\S 6.2]{Bate2005} generalises this part to non-connected reductive groups.
\end{proof}

\begin{remark}
After the reduction to the case $G = N$, the above proof of Theorem~\ref{THM:MAIN} is similar to the argument of \cite[Theorem 3.4]{Bate2008}, which shows that the $G$-complete reducibility of a subgroup $H \le G$ is preserved under taking quotients by subgroups of $H$ which are normal in $G$. Indeed, adapting that argument provides another proof of Theorem~\ref{THM:MAIN}, where one first reduces to the case $C \le H$ (rather than $H \le M$).
\end{remark}

\begin{remark}
If instead of $H \le N$ we assume that both $H$ and $K$ are reductive and $K^{\circ}$ normalises $H$, then $H$ is always relatively $G$-cr with respect to $K$. This follows from \cite[Corollary~3.28]{Bate2011}, since if $K$ normalises $H$ then in particular a maximal torus of $K$ normalises $H$, and this implies that $H$ is relatively $G$-cr with respect to $K$.
\end{remark}

Before discussing a number of consequences of Theorem~\ref{THM:MAIN} in the next section, we take this opportunity to note a slightly more general result, and to observe that some subtleties can arise when dealing with R-parabolic and R-Levi subgroups of disconnected groups. The following slight generalisation of Theorem~\ref{THM:MAIN} shows that we can factor out any normal subgroup of $C_{G}(K^{\circ})$ without affecting relative complete reducibility.

\begin{theorem} \label{thm:othersubs}
Keep the notation of Theorem~\ref{THM:MAIN} and let $M$ be as in Lemma~\ref{lem:usefulsub}. Let $f~\colon~N~\to~G'$ be a homomorphism into a reductive group $G'$ such that $\ker(f) \le C$. 
If $H \le M \ker(f)$, then $H$ is relatively $G$-completely reducible with respect to $K$ if and only if $f(H)$ is relatively $G'$-completely reducible with respect to $f(K)$.
\end{theorem}

\begin{proof} As in the proof of Theorem~\ref{THM:MAIN} it suffices to prove the result for the subgroup $H \ker(f) \cap M$ of $M$,  because this is contained in an R-parabolic or R-Levi subgroup corresponding to $\lambda \in Y(K)$ if and only if $H$ is since $\ker(f) \le C$. Moreover, the hypotheses imply that $f(H \ker(f) \cap M) = f(H)$. Also, since $\ker(f) \le C$, and since the restriction $\pi \colon M \to N/C$ gives an isogeny onto its image, the restriction $f \colon M \to N/\ker(f)$ is also an isogeny onto its image. Then as before we can reduce to the case that $G = N = MZ(K^{\circ})^{\circ}$, since this contains both $K^{\circ}$ and $H$, and this does not change $f(H)$ or the set of cocharacters of $f(K)$, hence does not change which R-parabolic subgroups or R-Levi subgroups of $G'$ stemming from $f(K)$ contain $f(H)$. Thus $f(G)$ is reductive and contains $f(K)$ and $f(H)$, and we can thus also assume that $f(G) = G'$.

So now $H$ is relatively $G$-cr with respect to $K^{\circ}$ precisely when $H$ is $G$-cr, and also $f(K^{\circ}) = (G')^{\circ}$ so $f(H)$ is relatively $G'$-cr with respect to $f(K)$ precisely when $f(H)$ is $G'$-cr. As before, the desired result reduces to the result proved in \cite[\S 6.2]{Bate2005}.
\end{proof}

\begin{remark}
When $\lambda \in Y(K) \setminus Y(M)$ it is not necessarily the case that $\pi(P_{\lambda} \cap N) = P_{\pi \circ \lambda}$ and $\pi(L_{\lambda} \cap N) = L_{\pi \circ \lambda}$, as the following example shows. So although relative complete reducibility behaves well with respect to taking quotients by subgroups centralising $K^{\circ}$, some care is required in the proofs.

Let $G$ be a connected reductive group with a maximal torus $T$ such that $T$ and $N_{G}(T)/T$ are non-trivial, and let $K = N = N_{G}(T)$. Then $K^{\circ} = C = T$ and $\pi(N) = N/C$ is the Weyl group of $G$, hence finite. For the subgroup $M$ of Lemma~\ref{lem:usefulsub} we can take any finite subgroup of $K$ which maps onto $\pi(N)$. Then $Y(M) = Y(\Kder)$ contains only the trivial cocharacter. Thus the only R-parabolic subgroup (and the only R-Levi subgroup) of $\pi(N)$ is $\pi(N)$ itself, and its pre-image under $\pi$, namely $N$, is not contained in any proper R-parabolic subgroup of $G$. Thus for any $\lambda \in Y(K)$ whose image is non-central in $G$, the images $\pi(P_{\lambda} \cap N)$ and $\pi(L_{\lambda} \cap N)$ are proper subgroups of $\pi(N)$ and are therefore not R-parabolic subgroups or R-Levi subgroups.
\end{remark}

\begin{remark} \label{rem:irred}
Recall that a subgroup $H$ of $G$ is called \emph{relatively $G$-irreducible} with respect to $K$ if, whenever $\lambda \in Y(K)$ such that $H \le P_{\lambda}$, we have $P_{\lambda} = G$ \cite[Definition 3.14]{Bate2011}. In this case, if $K = G$ then $H$ is called \emph{$G$-irreducible} \cite{Serre2003-2004}. We observe that the analogue of Theorem~\ref{THM:MAIN} does not hold for relative $G$-irreducibility. For instance if $K$ contains any non-central torus $S$ of $G$, then $C$ is contained in $C_{G}(S)$, a proper R-Levi subgroup of $G$, so $C$ is not relatively $G$-irreducible with respect to $K$. On the other hand, if we suppose that $\pi(N) = N/C$ is finite then every subgroup of $\pi(N)$ is $\pi(N)$-irreducible since $Y(\pi(N))$ is trivial, in particular the trivial subgroup (which is the image of $C$) is $\pi(N)$-irreducible in this case.
\end{remark}

The essential problem in the above discussion is the existence of non-trivial tori in the kernel of $N \to N/C$, i.e.\ a non-trivial torus in $Z(K^{\circ})$. The following result generalises \cite[Lemma 6.15]{Bate2005}, which considers the case that $N$ is reductive. This allows us to extend Theorem~\ref{THM:MAIN} to relative $G$-irreducibility, assuming $K^{\circ}$ is semisimple (Corollary \ref{cor:irred} below).

\begin{proposition} \label{prop:usefulsub} Keep the notation of Theorem~\ref{THM:MAIN} and let $M \le N$ be a reductive subgroup guaranteed by Lemma~\ref{lem:usefulsub}, so that $N = MC$ and $M \cap C$ is finite. Let $f \colon N \to G'$ be a homomorphism into a reductive group $G'$ such that $\ker(f) \le C$. If $\lambda \in Y(M)$ then
\begin{align*}
f(P_{\lambda} \cap N) &= f(P_{\lambda} \cap M)f(C) = P_{f \circ \lambda} \cap f(N), &
f^{-1}(P_{f \circ \lambda} \cap f(N)) &= P_{\lambda} \cap N, \\
f(L_{\lambda} \cap N) &= f(L_{\lambda} \cap M)f(C) = L_{f \circ \lambda} \cap f(N), &
f^{-1}(L_{f \circ \lambda} \cap f(N)) &= L_{\lambda} \cap N.
\end{align*}
In particular, taking $f = \pi \colon N \to N/C$ we have
\begin{align*}
\pi(P_{\lambda} \cap N) &= \pi(P_{\lambda} \cap M) = P_{\pi \circ \lambda}, &
\pi^{-1}(P_{\pi \circ \lambda}) &= P_{\lambda} \cap N, \\
\pi(L_{\lambda} \cap N) &= \pi(L_{\lambda} \cap M) = L_{\pi \circ \lambda}, &
\pi^{-1}(L_{\pi \circ \lambda}) &= L_{\lambda} \cap N.
\end{align*}
\end{proposition}

\begin{proof} Since $\ker(f) \le C$ and $\pi \colon M \to N/C$ is an isogeny onto its image, it follows that $f \colon M \to f(M)$ is an isogeny. Thus by the non-connected version of \cite[Lemma 2.11]{Bate2005} the R-parabolic subgroups and R-Levi subgroups of $f(M)$ are precisely the subgroups $P_{f \circ \lambda} \cap f(M)$ and $L_{f \circ \lambda} \cap f(M)$ for $\lambda \in Y(M)$, and these are respectively equal to $f(P_{\lambda} \cap M)$ and $f(L_{\lambda} \cap M)$. Since $C \le L_{\lambda} \le P_{\lambda}$ for all $\lambda \in Y(K)$, we also have $f(C) \le L_{f \circ \lambda} \le P_{f \circ \lambda}$, and since also $N = MC$ we have $f(N) = f(M)f(C)$ and it follows that
\[ f(P_{\lambda} \cap N) = f( (P_{\lambda} \cap M)C) = f(P_{\lambda} \cap M)f(C) = (P_{f \circ \lambda} \cap f(M))f(C) = P_{f \circ \lambda} \cap f(N), \]
and $f(L_{\lambda} \cap N) = f(L_{\lambda} \cap M)f(C) = L_{f \circ \lambda} \cap f(N)$ follows similarly. Finally, since $\ker(f) \le C$, we have
\[ f^{-1}(P_{f \circ \lambda} \cap f(N)) = f^{-1}f(P_{\lambda} \cap N) = (P_{\lambda} \cap N)\ker(f) = P_{\lambda} \cap N\]
and similarly for $L_{\lambda}$.
\end{proof}

The latter statements above show in particular that if $H$ is a subgroup of $G$, then $H \le P_{\lambda}$ or $L_{\lambda}$ for $\lambda \in Y(M)$ precisely when $\pi(H) \le P_{\pi \circ \lambda}$ or $L_{\pi \circ \lambda}$, respectively. In particular, if $Y(K) = Y(M)$ (i.e.\ if $K^{\circ}$ is semisimple), we obtain the following.

\begin{corollary} \label{cor:irred}
In the notation of Theorem~\ref{THM:MAIN}, if $K^{\circ}$ is semisimple then a subgroup $H$ of $N$ is relatively $G$-irreducible with respect to $K$ if and only if $\pi(H)$ is $\pi(N)$-irreducible.
\end{corollary}

We close this section with an extended example illustrating 
Theorem~\ref{THM:MAIN}.

\begin{example}
	Write $K = Cl(W)$ to indicate that $K$ is a special orthogonal or symplectic group with natural module $W$, in characteristic $p \ge 0$. Take an orthogonal direct sum $V = W_{1} \perp W_{2}$ where $W_{i} \cong W$ for $i = 1,2$, and let $G = Cl(V)$ (so $G$ has type $C_n$ or $D_n$ for some $n > 0$). We have a chain of subgroups
	\[ K \le Cl(W_{1}) \times Cl(W_{2}) \le Cl(V) = G, \]
	where the first embedding is just the diagonal one. Write $K_0$ for the left-hand group, and $K_i = Cl(W_{i})$ for $i = 1,2$.
	\begin{enumerate}
		\item If $p \neq 2$ then $K_{0}$ is $G$-cr, contained in a Levi subgroup of type $A_{n-1}$ corresponding to a direct-sum decomposition of $V$ into two totally isotropic $K_{0}$-submodules. Thus $C_{G}(K_{0})^{\circ}$ is a $1$-dimensional torus, consisting of elements acting as a scalar on $W_{1}$ and as the inverse scalar on $W_{2}$. \label{pnot2}
		\item If $p = 2$ then $K_{0}$ stabilises a unique nonzero totally isotropic subspace of $V$ \cite[Example 3.45]{Bate2005}, which is a diagonal submodule $W_{0} \subset W_{1} + W_{2}$. This shows that $K_{0}$ is contained in a parabolic subgroup of $G$ whose Levi factor has type $A_{n-1}$, but not in any Levi subgroup of $G$. In particular, $K_{0}$ is non-$G$-cr and $C_{G}(K_{0})^{\circ}$ is unipotent. In fact one can show that $C_{G}(K_{0})^{\circ}$ is a $1$-dimensional unipotent group; writing $V = W_{0} + W_{1}$ and identifying elements of $W_0$ and $W_1$ via a $K_{0}$-module isomorphism, this unipotent group consists of the maps $(w_{0},w_{1}) \mapsto (w_{0} + \lambda w_{1},w_{1})$ for $\lambda \in k$. The image of $K_{0}$ under projection to the Levi factor is a subgroup $K_{0}'$ stabilising a totally isotropic complement to $W_{0}$. Then $K_{0}'$ is $G$-cr, and its centraliser is simple of type $A_1$; $K_{0}'C_{G}(K_{0}')$ acts on $V$ as a tensor product $W_{0} \otimes V_2$, where $V_2$ is the natural $\SL_{2}(k)$-module. \label{p2-non-cr}
		\item If we repeat the above construction but with a $K$-module $W_{1} \perp W_{2}$ where $W_{2}$ is nontrivial and not isomorphic to $W_{1}$, then $K_{0}$ is $G$-irreducible and has trivial connected centraliser (independently of $p$). \label{irred-case}
	\end{enumerate}
	These three cases illustrate firstly that even a uniform (characteristic-independent) construction can lead to variation in $G$-complete reducibility and the structure of normalisers and centralisers. Part \ref{p2-non-cr} gives an example of a non-trivial unipotent subgroup of $G$, namely $C_{G}(K_{0})^{\circ}$, which is relatively $G$-cr with respect to $K_{0}$. The tables of \cite{Litterick2018} give further examples of unipotent groups (and more general groups with nontrivial unipotent radicals) arising as centralisers of reductive groups.
	
	In all three cases, Theorem~\ref{THM:MAIN} tells us that if $H \le N_{G}(K_{i})$ then $H$ is relatively $G$-cr with respect to $K_{i}$ if and only if the image of $H$ in $K_{i}/Z(K_{i})$ is completely reducible. The complete reducibility of this image can be  characterised purely in terms of the natural module $W_{i}$ if $i = 1$ or $2$. In part \ref{p2-non-cr}, $N_{G}(K_{0}) = K_{0}U_{1}$ is also non-$G$-cr, hence contained in the unique maximal parabolic subgroup of $G$ which contains $K_{0}$; in particular it stabilises $W_{0}$. Then a subgroup $H \le N_{G}(K_{0})$ acts on $W_{0}$, and to determine whether the image of $H$ in $N_{G}(K_{0})/C_{G}(K_{0})$ is completely reducible (hence whether $H$ is relatively $G$-cr with respect to $K_0$) we only need to consider the action of $H$ on $W_{0}$.
\end{example}

\section{Consequences of Theorem~\ref{THM:MAIN}} \label{sec:corollaries}

Armed with Theorem~\ref{THM:MAIN}, we deduce many core results for relative complete reducibility directly from their counterparts in the absolute setting. To begin, keeping the notation of Theorem~\ref{THM:MAIN}, it is clear that $\pi(N) = N/C$ and the trivial subgroup of $\pi(N)$ are $\pi(N)$-cr. Hence the following is immediate from Theorem~\ref{THM:MAIN}.
\begin{corollary} \label{cor:NORMAL-II}
Let $K \le G$ be reductive algebraic groups. Then $N_{G}(K^{\circ})$ and $C_{G}(K^{\circ})$ 
are relatively $G$-completely reducible with respect to $K$.
\end{corollary}

Since $K$ is $K$-cr, i.e.\ relatively $G$-cr with respect to itself, Corollary~\ref{cor:NORMAL-II} can be viewed as a generalisation of \cite[Corollary 3.16 and Corollary 3.17]{Bate2005}, which assert that both $N_{G}(K^{\circ})$ and $C_{G}(K^{\circ})$ are $G$-cr provided $K$ is. Note that $C_{G}(K^{\circ})$ is contained in $L_{\lambda}$ for all $\lambda \in Y(K)$ and is therefore clearly relatively $G$-cr with respect to $K$. However the conclusion for $N_{G}(K^{\circ})$ is less obvious. Note that $N_{G}(K^{\circ})$ and $C_{G}(K^{\circ})$ need not be reductive in general. When they are reductive, the assertion of Corollary~\ref{cor:NORMAL-II} follows from \cite[Corollary 3.28]{Bate2011}.

\subsection{Relative complete reducibility, reductivity and semisimple modules} As mentioned in the introduction, it is well-known that in the absolute case a $G$-cr subgroup is reductive. 
Furthermore, a combination of results of Jantzen \cite{Jantzen1997}, McNinch \cite{McNinch1998} and Liebeck and Seitz \cite{Liebeck1996} tells us that if $\Char(k)$ is sufficiently large and is coprime to $|H/H^{\circ}|$ for a closed subgroup $H$ of $G$, then $H$ is $G$-cr if and only if $H$ is reductive. 
More specifically, for a simple algebraic group $X$ define $a(X)$ to be $1$ plus the rank of $X$ (the dimension of a maximal torus of $X$). 
For a general reductive group $X$, let $a(X)$ be the maximum value of $a(Y)$ over all simple factors $Y$ of $X$. 
Then \cite[Th\'{e}or\`{e}me 4.4]{Serre2003-2004} states that if $H$ is a subgroup of $G$, and if $\Char(k)$ is zero or a prime $p \ge a(G)$ which is coprime to $|H/H^{\circ}|$, then $H$ is $G$-cr if and only if it is reductive. 
The relative version of this result, of which Corollary~\ref{cor:NORMAL} is a special case, is as follows.
\begin{theorem} \label{thm:bigp} In the notation of Theorem~\ref{THM:MAIN}, the following hold.
\begin{enumerate}
	\item If $H$ is relatively $G$-completely reducible with respect to $K$ then $R_{u}(H) \le C$.
	\item If $\Char(k)$ is zero or a prime $p \ge a(K)$ not dividing $|\pi(H)/\pi(H)^{\circ}|$, then $H$ is relatively $G$-completely reducible with respect to $K$ if and only if $R_{u}(H) \le C$.
\end{enumerate}
\end{theorem}
\begin{proof} For (i), Theorem~\ref{THM:MAIN} together with the absolute result \cite[Property 4]{Serre1998} implies that if $H$ is relatively $G$-cr with respect to $K$, then $\pi(H)$ is reductive, hence $R_{u}(H) \le C$. For (ii), if $R_{u}(H) \le C$ then $\pi(H)$ is reductive. Moreover $\Char(k)$ is either zero or coprime to $|\pi(H)/\pi(H)^{\circ}|$. Then the absolute result \cite[Th\'{e}or\`{e}me 4.4]{Serre2003-2004} tells us that $\pi(H)$ is $\pi(N)$-cr, and from Theorem~\ref{THM:MAIN} we see that $H$ is relatively $G$-cr with respect to $K$.
\end{proof}

Theorem~\ref{thm:bigp} gives an intrinsic group-theoretic characterisation of relative $G$-complete reducibility in characteristic zero, generalising the result from the absolute setting. Next, complete reducibility in the absolute case is closely linked to the semisimplicity of $G$-modules. More precisely, let $T$ be a maximal torus of $G$, let $\Phi^{+}$ be a choice of positive roots of $G$ with respect to $T$, and for a $G$-module $V$ define $n(V) = \max\left\{ \sum_{\alpha \in \Phi^{+}} \left<\lambda,\alpha^{\vee}\right>\right\}$, the maximum over $T$-weights $\lambda$ of $V$. Then a result of Serre \cite[Th\'{e}or\`{e}me 5.4]{Serre2003-2004} states that if $H$ is $G$-cr and $\Char(k)$ is zero or greater than $n(V)$, then $V$ is semisimple as an $H$-module. Moreover if $V$ is \emph{non-degenerate} (i.e.\ the identity component of the kernel is a torus), the converse also holds. Since $n(\Lie(G)) = 2 h_{G} - 2$, where $h_{G}$ is the Coxeter number of $G$, this gives a concrete criterion for $G$-complete reducibility in terms of the action of $H$ on $\Lie(G)$ \cite[Corollary 5.5]{Serre2003-2004}. The relative versions of these results are as follows.

\begin{theorem} \label{thm:module} In the notation of Theorem~\ref{THM:MAIN}, the following hold.
\begin{enumerate}
	\item Let $V$ be a $\pi(N)$-module, and suppose that $\Char(k)$ is zero or greater than $n(V)$. If $H$ is relatively $G$-completely reducible with respect to $K$, then $V$ is semisimple as a $\pi(H)$-module. Conversely, if $V$ is non-degenerate and semisimple as a $\pi(H)$-module, then $H$ is relatively $G$-completely reducible with respect to $K$.
	\item If $\Char(k)$ is zero or greater than $2h_{K} - 2$, where $h_{K}$ is the Coxeter number of $K$, then $H$ is relatively $G$-completely reducible with respect to $K$ if and only if $\Lie(K)$ is semisimple as an $H$-module.
\end{enumerate}
\end{theorem}

\begin{proof} Part (i) follows directly from Theorem~\ref{THM:MAIN} and the result in the absolute case applied to $N/C$. For (ii), note that since $N$ normalises $K^{\circ}$ it acts on the adjoint module $\Lie(K)$, hence so does $H$. Moreover this action factors through $\pi(N)$ since $C$ centralises $K^{\circ}$. Now the isomorphism $\pi(N)^{\circ} \cong K^{\circ}/Z(K^{\circ})$ implies that $h_{K} = h_{\pi(N)}$. Therefore, since $\Lie(K)$ is non-degenerate as an $(N/C)$-module, part (ii) follows from part (i) and the fact that $n(\Lie(K)) = 2h_{K} - 2$.
\end{proof}

Putting Theorem~\ref{thm:bigp}(ii) and Theorem~\ref{thm:module}(ii) together, and using the fact that $\pi(H)/\pi(H^{\circ})$ is a quotient of $H/H^{\circ}$, gives the following.

\begin{corollary}
Let $K \le G$ be reductive algebraic groups, let $H \le N_{G}(K^{\circ})$ and suppose that $\Char(k)$ is zero or a prime $p \ge 2h_{K} - 2$ which does not divide $|H/H^{\circ}|$. Then the following are equivalent:
\begin{enumerate}
	\item $H$ is relatively $G$-completely reducible with respect to $K$;
	\item $\Lie(K)$ is semisimple as an $H$-module;
	\item $R_{u}(H)$ centralises $K^{\circ}$.
\end{enumerate}
\end{corollary}

A result of Jantzen~\cite[Proposition 3.2]{Jantzen1997} states that if $G$ is connected and $V$ is a $G$-module, and if $\dim V \le p$ when $\Char(k) = p > 0$, then $V$ is semisimple. In \cite{Bate2011a} this is generalised to show that $V$ is also semisimple as an $H$-module for every $G$-cr subgroup $H$ of $G$. The relative variant of this result is as follows; this follows directly from the absolute result \cite[Theorem 1.3]{Bate2011a} and Theorem~\ref{THM:MAIN}.
\begin{corollary}
Keep the notation of Theorem~\ref{THM:MAIN}. Suppose that $\pi(N)$ is connected and let $V$ be a $\pi(N)$-module. If $\Char(k) = p > 0$, assume that $\dim V \le p$.

If $H \le N$ is relatively $G$-completely reducible with respect to $K$ then $V$ is semisimple as an $H$-module.
\end{corollary}

The following result also follows directly from its absolute counterpart \cite[Theorem 1.4]{Bate2011a} and Theorem~\ref{THM:MAIN}. Recall that a module $V$ for an algebraic group is called \emph{non-degenerate} if the identity component of the kernel of the action is a torus.
\begin{corollary} \label{cor:tensor-cr-g-cr}
With the notation of Theorem~\ref{THM:MAIN}, suppose that $\pi(N)$ is connected and let $V$ be a non-degenerate $\pi(N)$-module. If $V \otimes V^{\ast}$ is semisimple as a $\pi(H)$-module then $H$ is relatively $G$-completely reducible with respect to $K$.
\end{corollary}

\begin{remark}
If $V$ is a non-degenerate $\pi(N)$-module then $V$ is a non-degenerate $K$-module, since the identity components of the kernels of $K \to \pi(K) = K/C_{K}(K^{\circ})$ and $\pi(N) \to \GL(V)$ are both tori. Therefore in the particular case $H \le K$, noting that $H$ is relatively $G$-cr with respect to $K$ precisely when $H$ is $K$-cr, Corollary~\ref{cor:tensor-cr-g-cr} specialises to the complete reducibility statement of \cite[Theorem 1.4]{Bate2011a}.
\end{remark}

\subsection{Normal subgroups, normalisers, Clifford theory} Recall that in the absolute setting, a normal subgroup of a $G$-cr subgroup is again $G$-cr, and furthermore a subgroup $H \le G$ is $G$-cr if and only if $N_{G}(H)$ is $G$-cr \cite[Theorem 3.10 and Corollary 3.16]{Bate2005}; this generalises Clifford's Theorem from representation theory. Moreover if $H$ is $G$-cr and $H'$ is a subgroup of $G$ which satisfies $HC_{G}(H)^{\circ} \le H' \le N_{G}(H)$ then $H'$ is $G$-cr \cite[Theorem 3.14]{Bate2005}. The relative versions of these results are as follows.

\begin{theorem} \label{thm:clifford}
With the notation of Theorem~\ref{THM:MAIN}, the following hold.
\begin{enumerate}
	\item Let $H'$ be a normal subgroup of $H$. If $H$ is relatively $G$-completely reducible with respect to $K$ then so is $H'$. \label{cliff-i}
	\item Let $H'$ be a subgroup of $N$ satisfying $(HC)N_{N}(HC)^{\circ} \le H'C \le N_{N}(HC)$. If $H$ is relatively $G$-completely reducible with respect to $K$, then so is $H'$. \label{cliff-ii}

	\item $H$ is relatively $G$-completely reducible with respect to $K$ if and only if $N_{N}(HC)$ is relatively $G$-completely reducible with respect to $K$. \label{cliff-iii}
	\item If $H$ is relatively $G$-completely reducible with respect to $K$, then so is $C_{N}(HC)$. \label{cliff-iv}
\end{enumerate}
\end{theorem}

\begin{proof} Part \ref{cliff-i} follows directly from Theorem~\ref{THM:MAIN} and the absolute result \cite[Theorem 3.10]{Bate2005}, since $\pi(H')$ is normal in $\pi(H)$.

For part \ref{cliff-ii}, if $H$ is relatively $G$-cr with respect to $K$ then by Theorem~\ref{thm:bigp}, $\pi(H)$ is reductive. Therefore $N_{\pi(N)}(\pi(H))^{\circ} = \pi(H)^{\circ} C_{\pi(N)}(\pi(H))^{\circ}$. The group $\pi(N_{N}(HC)^{\circ})$ is a connected subgroup of $N_{\pi(N)}(\pi(H))$ containing $\pi(H)^{\circ}$. Moreover, if $\pi(x) \cdot \pi(h) = \pi(h)$ for some $h \in H$ then $x \cdot h = hc'$ for some $c' \in C$. Since $C$ is normal in $N$, it follows that $x \cdot (hc) \in HC$ for all $h \in H$ and all $c \in C$. Hence $N_{N}(HC)^{\circ} \ge \pi^{-1}(C_{\pi(N)}(\pi(H))^{\circ})$, and so $\pi(HCN_{N}(HC)^{\circ}) = \pi(H)C_{\pi(N)}(\pi(H))^{\circ}$. Since we also have $\pi(N_{N}(H)) \le N_{\pi(N)}(\pi(H))$, the hypotheses imply that $\pi(H)C_{\pi(N)}(\pi(H))^{\circ} \le \pi(H') \le N_{\pi(N)}(\pi(H))$. Applying the result from the absolute case \cite[Theorem 3.14]{Bate2005}, we conclude that $\pi(H')$ is $\pi(N)$-cr, and so $H'$ is relatively $G$-cr with respect to $K$.

For part \ref{cliff-iii}, if $H$ is relatively $G$-cr with respect to $K$ then by \ref{cliff-ii} so is $N_{N}(HC)$. Conversely if $N_{N}(HC)$ is relatively $G$-cr with respect to $K$ then by \ref{cliff-i} so is its normal subgroup $HC$. This is contained in precisely the same parabolic and Levi subgroups corresponding to elements of $Y(K)$ as $H$, and so $H$ is also relatively $G$-cr with respect to $K$. Part \ref{cliff-iv} now follows from parts \ref{cliff-i} and \ref{cliff-iii}, since $C_{N}(HC)$ is normal in $N_{N}(HC)$.
\end{proof}

\begin{remark}
	Theorem~\ref{thm:clifford} fails without the hypothesis that $H \le N$, even if we impose other natural conditions, for instance requiring that $H$ is connected and reductive \cite[Examples 5.6 and 5.7]{Bate2011}.
\end{remark}

The following example demonstrates the failure of Theorem~\ref{thm:bigp}(i) and Theorem~\ref{thm:clifford}(i) 
if $H$ does not normalize $K^\circ$, even when $H$ is connected.

\begin{example}
	Let $G=\GL_n(k)$, $K=\SO_n(k)$ and let $e_1,\ldots,e_n$ be the standard basis of $k^n$.
	Suppose that $\text{char}(k)\geq 3$ or $n\geq 3$.
	Let $H$ be the stabilizer in $G$ of the totally isotropic subspace $\langle e_1\rangle$. Since $H$ is relatively $G$-irreducible with respect to $K$, it is 
	relatively $G$-completely reducible with respect to $K$. But its normal subgroup $R_u(H)$ is not relatively $G$-completely reducible with respect to $K$, since $R_u(H)\le R_u(\text{Stab}_G(f))$, where $f$ is the flag $f= \langle e_1 \rangle\le \langle e_1\rangle^ \perp$, and $\text{Stab}_G(f)$ is a parabolic subgroup of $G$ which is given by a cocharacter of $K$.
	
	In addition, $R_u(H)$ is not contained in $N_G(K^\circ)$.
	Thus $H$ does not normalize $K^\circ$ and $R_u(H)$ does not centralize $K^\circ$.
	For let $h=\left( \begin{array}{rrrr}
	1 & 1 & \cdots & 1 \\
	0 & \ddots & 0 & \vdots \\
	\vdots & 0 & \ddots & 1 \\
	0 & \cdots & 0 & 1 \\
	\end{array}\right) \in R_u(H)$. 
	Since $\text{char}(k)\geq 3$ or $n\geq 3$, we see that $h t h^{-1}\not\in K$, for $1 \ne t \in \text{D}_n\cap K$, where $\text{D}_n$ is the subgroup of diagonal matrices in $G$.
\end{example}

The following directly generalises \cite[Theorem 1.3]{Bate2008}.
\begin{corollary}
With the notation of Theorem~\ref{THM:MAIN}, 
suppose that $\pi(N)$ is connected and that $p > 3$ or $p$ is good for $K$. Let $A$, $B$ be commuting connected subgroups of $N$ which are relatively $G$-completely reducible with respect to $K$. Then $AB$ is relatively $G$-completely reducible with respect to $K$.
\end{corollary}
\begin{proof} This follows directly from the absolute result \cite[Theorem 1.3]{Bate2008} and Theorem~\ref{THM:MAIN}, since $\pi(AB)$ is the commuting product of the connected groups $\pi(A)$ and $\pi(B)$.
\end{proof}

\subsection{A geometric viewpoint} We end this section with a geometric consequence of Theorem~\ref{THM:GEOMETRIC}.

\begin{theorem}
\label{THM:geom}
In the notation of Theorem~\ref{THM:GEOMETRIC}, with $\mathbf{h} \in N^{n}$ and $\pi$ also denoting the quotient map $N^{n} \to (N/C)^{n}$, the following are equivalent.
\begin{enumerate}
	\item $K \cdot \mathbf{h}$ is closed in $G^{n}$. \label{geom-i}
	\item $\pi(K) \cdot \pi(\mathbf{h})$ is closed in $(N/C)^{n}$. \label{geom-ii}
	\item $\pi^{-1}(\pi(K \cdot \mathbf{h}))$ is closed in $G^{n}$. \label{geom-iii}
	\item Every orbit of $K$ on $\pi^{-1}(\pi(K \cdot \mathbf{h}))$ is closed in $G^{n}$. \label{geom-iv}
\end{enumerate}
\end{theorem}

\begin{proof} Theorem~\ref{THM:GEOMETRIC} tells us that \ref{geom-i} and \ref{geom-ii} are equivalent. Now $N^{n}$ is closed in $G^{n}$ and, since $\pi(K) \cdot \pi(\mathbf{h}) = \pi(K \cdot \mathbf{h})$ and the topology on $(N/C)^{n}$ is the quotient topology, it follows that \ref{geom-ii} and \ref{geom-iii} are equivalent. 

Let $K = \bigcup_{i=1}^t K^\circ x_i$ be the decomposition of $K$ into right cosets of $K^\circ$. Expressing $\mathbf{h}$ as $\mathbf{h} = (h_1,\ldots,h_n)$, we have
\begin{align*}
\pi^{-1}(\pi(K \cdot \mathbf{h})) &= \bigcup_{i=1}^t \pi^{-1}(\pi \left(K^\circ \cdot (x_i \cdot \mathbf{h}) \right)\\
&= \bigcup_{i=1}^t \left\{ \left(x \cdot (x_i \cdot h_1)c_1,\ldots, x \cdot (x_i \cdot h_n)c_n\right) \ssep x \in K^\circ,\ c_{i} \in C \right\},\\
&= \bigcup_{i=1}^t \left\{ \left(x \cdot ((x_i \cdot h_1)c_1),\ldots, x \cdot ((x_i \cdot h_n)c_n)\right) \ssep x \in K^\circ,\ c_{i} \in C \right\},
\end{align*}
where the last equality follows since
$K^\circ$ centralises $C$.

Now if $K \cdot \mathbf{h}$ is closed in $G^{n}$, then so is $K^\circ \cdot (x_i \cdot \mathbf{h})$ for each $i$.
Thus $\pi^{-1}(\pi(K\cdot \mathbf{h}))$ is a union of $K^\circ$-orbits, each of which is translated to one of the closed $K^\circ$-orbits 
$K^\circ \cdot (x_i \cdot \mathbf{h})$ by
an element of $C^{n}$.
As translation is a variety automorphism $G^{n} \to G^{n}$, it follows that every $K^\circ$-orbit in $\pi^{-1}(\pi(K \cdot \mathbf{h}))$ is closed in $G^{n}$.
Consequently, every $K$-orbit in $\pi^{-1}(\pi(K \cdot \mathbf{h}))$ is closed in $G^{n}$ as well.
So \ref{geom-iv} follows from \ref{geom-i} and the reverse implication is clear.
\end{proof}

\begin{remark}
\label{rem:geometric}
The proofs of the equivalences \ref{geom-ii} $\Leftrightarrow$ \ref{geom-iii} and \ref{geom-i} $\Leftrightarrow$ \ref{geom-iv} in Theorem~\ref{THM:geom} are easily seen to be independent of Theorems~\ref{THM:MAIN} and \ref{THM:GEOMETRIC}. From the argument above, we see that $\pi^{-1}(\pi(K \cdot \mathbf{h}))$ consists of $K^\circ $-orbits which are $C^{n}$-translates of a $K^\circ $-orbit 
$K^\circ \cdot (x_i \cdot \mathbf{h})$ for some $i$. 
The observation that one of the latter is closed if and only if all of them are closed, together with the fact that closed orbits always exist, we conclude that all $K^\circ$-orbits are closed in $\pi^{-1}(\pi(K \cdot \mathbf{h}))$ in the subspace topology and thus every $K$-orbit in $\pi^{-1}(\pi(K \cdot \mathbf{h}))$ is closed in $\pi^{-1}(\pi(K \cdot \mathbf{h}))$ in the subspace topology. Hence we also see that \ref{geom-iii} implies \ref{geom-iv}, even without appealing to Theorem~\ref{THM:GEOMETRIC} (or Theorem~\ref{THM:MAIN}). However the final implication is more subtle since it is not clear \emph{a priori} that in this setting an arbitrary union of $C^{n}$-translates of a closed $K$-orbit is again Zariski-closed in $G^{n}$.
\end{remark}

Corollary \ref{cor:irred} also has a geometric counterpart, as follows. If $\mathbf{h} \in G^{n}$ is a generic tuple for a subgroup $H \le G$ then $H$ is relatively $G$-irreducible with respect to $K$ precisely when $\mathbf{h}$ is a \emph{$K$-stable point} of $G^{n}$, i.e.\ $K \cdot \mathbf{h}$ is closed in $G^{n}$ and $C_{K}(\mathbf{h})/C_{K}(G)$ is finite \cite[Definition 3.12, Proposition 3.16]{Bate2011}.

\begin{corollary} \label{cor:irredgeom}
In the notation of Theorem~\ref{THM:GEOMETRIC}, suppose that $K^{\circ}$ is semisimple and that $\mathbf{h} \in N^{n}$. Then the following are equivalent.
\begin{enumerate}
	\item $\mathbf{h}$ is a $K$-stable point of $G^{n}$. \label{cor-irr-i}
	\item $\pi(\mathbf{h})$ is a $\pi(K)$-stable point of $(N/C)^{n}$. \label{cor-irr-ii}
	\item $\pi^{-1}(\pi(K \cdot \mathbf{h}))$ is closed in $G^{n}$ and contains a $K$-stable point. \label{cor-irr-iii}
	\item Every point of $\pi^{-1}(\pi(K \cdot \mathbf{h}))$ is $K$-stable in $G^{n}$. \label{cor-irr-iv}
\end{enumerate}
\end{corollary}

\begin{proof} Corollary~\ref{cor:irred} gives the equivalence \ref{cor-irr-i} $\Leftrightarrow$ \ref{cor-irr-ii}. Theorem~\ref{THM:geom} tells us that the closure conditions in \ref{cor-irr-i}, \ref{cor-irr-iii} and \ref{cor-irr-iv} are equivalent. It therefore suffices to show that all $K$-orbits in $\pi^{-1}(\pi(K \cdot \mathbf{h}))$ have the same dimension, since then their stabilisers (in $K$) all have the same dimension, so $\pi^{-1}(\pi(K \cdot \mathbf{h}))$ contains a $K$-stable point if and only if all of its points are $K$-stable. As in the proof of Theorem~\ref{THM:geom}, the set $\pi^{-1}(\pi(K \cdot \mathbf{h}))$ is a union of $K^{\circ}$-orbits, each of which is a $C^{n}$-translate of a $K^{\circ}$-orbit $K^{\circ}x_{i} \cdot \mathbf{h}$, where $\{x_1,\ldots,x_t\}$ is a set of coset representatives for $K^{\circ}$ in $K$. But $K/K^{\circ}$ acts transitively by conjugation on these orbits since $K^{\circ}x_{i} \cdot \mathbf{h} = x_{i}K^{\circ} \cdot \mathbf{h} = x_{i} \cdot (K^{\circ} \cdot \mathbf{h})$, hence these orbits indeed have the same dimension.
\end{proof}

\section{Relative complete reducibility and separability} \label{sec:separabilty}

Recall from \cite[Definition 3.27]{Bate2005} that a closed subgroup $H$ of $G$ is called \emph{separable in $G$} if the Lie algebra centraliser $C_{\Lie(G)}(H)$ equals the Lie algebra of $C_{G}(H)$. As discussed in \cite[\S 3.5]{Bate2005}, this is equivalent to the smoothness of the scheme-theoretic centraliser of $H$ in $G$. Moreover if $H$ is topologically generated by $\{h_1,\ldots,h_n\}$ then $H$ is separable in $G$ precisely when the orbit map $G \to G \cdot (h_1,\ldots,h_n)$ is a separable morphism of varieties.

In \cite[\S 3.5]{Bate2005} and \cite{Bate2010} it is shown that separability interacts closely with complete reducibility. We now derive relative analogues of these results.

\begin{definition} \label{def:relKsep}
For subgroups $H$ and $K$ of a reductive group $G$, we say that $H$ is \emph{separable for $K$} if $\Lie(C_{K}(H))$ and $C_{\Lie(K)}(H)$ coincide as subspaces of $\Lie(G)$.
\end{definition}
This definition is equivalent to requiring that the orbit map $K \to K \cdot \mathbf{h}$ is a separable morphism for some (equivalently any) generic tuple $\mathbf{h}$ of $H$. This equivalence is proved in \cite[Lemma 5.1]{Bate2011a} under the assumption that $H \le K$ and $K$ is connected, but the same proof applies word-for-word in this more general setting. Note also that when $K = G$ this gives the usual definition of separability of a subgroup in $G$, cf.~\cite[Definition 3.27]{Bate2005}.

Recall from \cite[Theorem 3.5]{Bate2011} that a subgroup of $G$ is relatively $G$-cr with respect to $K$ if and only if for all $\lambda \in Y(K)$ such that $H \le P_{\lambda}$, we have $\dim C_{K}(H) = \dim C_{K}(c_{\lambda}(H))$, where $c_{\lambda}$ is the map $P_{\lambda} \to L_{\lambda}$, $x \mapsto \lim_{a \to 0}(\lambda(a) \cdot x)$. This observation allows us to prove the following analogue of \cite[Theorem 3.46]{Bate2005}.

\begin{theorem} \label{thm:sep_rel_gcr}
In the notation of Theorem~\ref{THM:MAIN}, suppose that $H$ is separable for $K$. If $\Lie(K)$ is semisimple as an $H$-module, then $H$ is relatively $G$-completely reducible with respect to $K$.
\end{theorem}

\begin{proof} We mimic the proof of \cite[Theorem 3.46]{Bate2005}. Suppose that $H$ is not relatively $G$-cr with respect to $K$. Thus by \cite[Theorem 3.5]{Bate2011} there exists $\lambda \in Y(K)$ such that $H \le P_{\lambda}$ and $\dim C_{K}(H) < \dim C_{K}(c_{\lambda}(H))$. Since $H$ is separable for $K$ by hypothesis, it follows that
\[ \dim C_{\Lie(K)}(H) = \dim C_{K}(H) < \dim C_{K}(c_{\lambda}(H)) \le \dim C_{\Lie(K)}(c_{\lambda}(H)).\]

Note that 
$c_\lambda(H)$ is in $N$:
for each $a$ in $k$, and $x$ in $N$, we have that $\lambda(a) x \lambda(a)^{-1}$ belongs to $N$. 
As $N$ is closed, the limit of the former as $a$ tends to $0$ still belongs to $N$. In particular, $c_\lambda(H)$ belongs to $N$.   

Now consider the images of $H$ and $c_{\lambda}(H)$ under the map $\Ad \colon N \to \GL(\Lie(K))$. Then it is clear that $\Ad(c_{\lambda}(H)) = c_{\Ad \circ \lambda}(H)$. Since $\Lie(K)$ is semisimple as an $H$-module, we deduce that $\Ad(H)$ is $\GL(\Lie(K))$-completely reducible, and in particular it is $\GL(\Lie(K))$-conjugate to $\Ad(c_{\lambda}(H))$. The fixed points of $\Ad(H)$ and $\Ad(c_{\lambda}(H))$ on $\Lie(K)$ are precisely $C_{\Lie(K)}(H)$ and $C_{\Lie(K)}(c_{\lambda}(H))$, and so these have equal dimensions. This contradicts the strict inequality above.
\end{proof}

\begin{remark}
One may be tempted to prove Theorem~\ref{thm:sep_rel_gcr} more directly by working with the image in $\pi(N)$ and applying Theorem~\ref{THM:MAIN} and the absolute result \cite[Theorem 3.46]{Bate2005}. However, even in the absolute case, separability of $H$ does not imply separability of $\pi(H)$ in $\pi(N)$. As example of this phenomenon, it is well-known that every subgroup of $\GL_{n}(k)$ is separable, whereas $\PGL_{n}(k)$ has non-separable subgroups, for instance the normaliser of a maximal torus in $\PGL_{2}(k)$ is not separable when $\Char(k) = 2$, cf.\ \cite[Examples 3.28--3.30]{Bate2005}.
\end{remark}

\begin{remark} \label{rem:torus}
In \cite[Theorem 1.2]{Bate2010} it is shown that if $G$ is connected and $\Char(k)$ is very good for $G$, then every subgroup of $G$ is separable in $G$. This does not generalise to non-connected $G$. Again, the normaliser of a maximal torus in $\PGL_{2}(k)$ provides a counterexample when $\Char(k) = 2$. This subgroup has the form $G = T \rtimes \left<x\right>$ where $x$ is an involution inverting every element of the $1$-dimensional torus $T$. Then $G$ is centreless but acts trivially on the $1$-dimensional Lie algebra $\Lie(G)$, in particular $G$ is not separable as a subgroup of itself, although $2$ is very good for $G$ as the root system is trivial (cf.\ \cite[Remark 3.5(iv)]{Bate2010}, which is missing the necessary condition $p = 2$).
\end{remark}

It turns out that the above is essentially the only obstruction, arising because $\Char(k)$ divides the order of the finite group $T \rtimes \left<x\right>/T$. The following generalises \cite[Theorem 1.2]{Bate2010} both to non-connected $G$ and to the relative setting.
\begin{theorem} \label{thm:vg_sep}
With the notation of Theorem~\ref{THM:MAIN}, suppose that $\Char(k)$ is zero or is very good for $K$ and coprime to $|\pi(H)/\pi(H \cap N^{\circ})|$. Then $H$ is separable for $K$. In particular, if $\Char(k)$ is zero or is very good for $K$ and coprime to one of $|\pi(H)/\pi(H^{\circ})|$ or $|\pi(N)/\pi(N^{\circ})|$, then $H$ is separable for $K$.
\end{theorem}

\begin{proof} For brevity, write $H' = H \cap N^{\circ}$. Note firstly that $H/H' \cong HN^{\circ}/N^{\circ} \le N/N^{\circ}$, and since $H^{\circ} \le H'$ we have a surjection $H/H^{\circ} \to H/H'$. Applying $\pi$, if $\Char(k)$ is positive and coprime to one of $|\pi(H)/\pi(H^{\circ})|$ or $|\pi(N)/\pi(N^{\circ})|$, then it is coprime to $|\pi(H)/\pi(H')|$. Thus it suffices to prove the first statement of the theorem.

Now, $C_{K}(H)^{\circ}$ is the largest connected subgroup of $K$ centralising $H$, hence is the largest connected subgroup of $K^{\circ}$ centralising $H$, thus $\dim C_{K}(H) = \dim C_{K^{\circ}}(H)$. Since also $\Lie(K) = \Lie(K^{\circ})$ it follows that we can assume $K = K^{\circ}$. Next let $M$ be a reductive subgroup of $N$ guaranteed by Lemma~\ref{lem:usefulsub}. It is clear that $\dim C_{K}(H) = \dim C_{K}(HC)$ and $C_{\Lie(K)}(H) = C_{\Lie(K)}(HC)$, and since $HC = (HC \cap M)C$ it suffices to prove the result assuming that $H \le M$ (this does not change $\pi(H)$ or $\pi(H \cap N^{\circ})$). We can therefore also assume $G = N = MZ(K)^{\circ}$, since this does not change $K$, $C_{K}(H)$, $C_{\Lie(K)}(H)$, $\pi(H)$ or $\pi(H \cap N^{\circ})$. In this case $H$ is separable for $K$ precisely when $H$ is separable in $G$.

We now have $G^{\circ} = N^{\circ} = K$, hence $\Char(k)$ is very good for $G$ by hypothesis. Now $H' = H \cap N^{\circ} = H \cap G^{\circ}$, in particular $H' \le G^{\circ}$ and by \cite[Theorem 1.2]{Bate2010} the subgroup $H'$ is separable in $G^{\circ}$. The finite group $H/H'$ acts on $C_{K}(H')$ and on $C_{\Lie(K)}(H')$, and $H$ is separable in $G$ (hence separable for $K$) precisely when the fixed points of $H/H'$ on $\Lie C_{K}(H') = C_{\Lie(K)}(H')$ are equal to the Lie algebra of $C_{C_{K}(H')}(H/H')$. But the action of $H$ on $K$ factors through $\pi(H)$, and so the action of $H/H'$ factors through $\pi(H)/\pi(H')$. By hypothesis, $\Char(k)$ is either zero or coprime to the order of this latter group, hence $\pi(H)/\pi(H')$ is linearly reductive and the desired result follows from \cite[Lemma 4.1]{Richardson1982}.
\end{proof}

Combining Theorem~\ref{thm:vg_sep} and Theorem~\ref{thm:sep_rel_gcr} gives the following, which in turn implies Theorem~\ref{THM:separable}.
\begin{corollary} \label{cor:vg_sep}
With the notation of Theorem~\ref{THM:MAIN}, suppose that $\Char(k)$ is zero or is very good for $K$ and coprime to $|\pi(H)/\pi(H \cap N^{\circ})|$. If $\Lie(K)$ is semisimple as an $H$-module, then $H$ is relatively $G$-completely reducible with respect to $K$.
\end{corollary}

In \cite[\S 3.5]{Bate2005} and \cite{Bate2010} it is shown that separability interacts closely with Richardson's notion of reductive pairs. We now generalise these results to the relative setting.

\begin{definition} \label{def:reductivepair}
Let $G$ be reductive, let $H,K \le G$ and suppose that $H$ normalises $K^{\circ}$. We say that $(G,K)$ is a \emph{reductive pair for $H$} if $\Lie(K)$ is an $H$-module direct summand of $\Lie(G)$.
\end{definition}
This generalises the usual notion of a reductive pair \cite[Definition 3.32]{Bate2005}, which is the special case $H = K$.

The following result generalises \cite[Corollary 5.3]{Bate2011a}, which is a corollary of \cite[Lemma 5.2]{Bate2011a}. The proof of this latter result does not use the assumption `$H \le K$', and therefore goes through word-for-word in this situation.
\begin{lemma}[cf.\ {\cite[Corollary 5.3]{Bate2011a}}] \label{lem:sep_red_sep}
Let $K \le G$ be reductive algebraic groups and let $H$ be a subgroup of $G$. If $H$ is separable for $G$ and $(G,K)$ is a reductive pair for $H$ then $H$ is separable for $K$.
\end{lemma}

The following corollary of Lemma~\ref{lem:sep_red_sep} generalises \cite[Corollary 5.4]{Bate2011a} to the case that $G$ is not necessarily connected.
\begin{corollary} \label{cor:sep_red_sep}
Take $G$, $K$, $H$ as in Lemma~\ref{lem:sep_red_sep} and let $\psi \colon G \to \Aut(G^{\circ})$ denote the map induced by conjugation. If $(G,K)$ is a reductive pair for $H$ and $\Char(k)$ is zero, or is very good for $G$ and coprime to $|\psi(H)/\psi(H \cap G^{\circ})|$, then $H$ is separable for $K$.
\end{corollary}

\begin{proof}
Apply Theorem~\ref{thm:vg_sep} (taking $K = G$ there) to conclude that $H$ is separable for $G$, and then apply Lemma \ref{lem:sep_red_sep}.
\end{proof}

\begin{proof}[Proof of Theorem~\ref{thm:GLV}] \label{prf:GLV}
It is well-known that every subgroup of $\GL(V)$ is separable in $\GL(V)$, hence Lemma~\ref{lem:sep_red_sep} implies that $H$ is separable for $K$. Again, since $\Lie(K)$ is a semisimple $H$-module, the desired conclusion follows from Theorem~\ref{thm:sep_rel_gcr}.
\end{proof}

Combining Corollary~\ref{cor:sep_red_sep} and Theorem~\ref{thm:sep_rel_gcr} gives the following, which in turn implies Theorem~\ref{thm:G-connected-red-pair}.
\begin{corollary} \label{cor:sep_red_sep_gcr}
Let $K \le G$ be reductive algebraic groups, let $H$ be a subgroup of $G$ and let $\psi \colon G \to \Aut(G^{\circ})$ denote the map induced by conjugation. Suppose that $(G,K)$ is a reductive pair for $H$ and $\Char(k)$ is either zero, or is very good for $G$ and coprime to $|\psi(H)/\psi(H \cap G^{\circ})|$. If $\Lie(K)$ is semisimple as an $H$-module then $H$ is relatively $G$-completely reducible with respect to $K$.
\end{corollary}

\begin{remark}
Theorem~\ref{thm:vg_sep}, Corollary~\ref{cor:sep_red_sep} and Corollary~\ref{cor:sep_red_sep_gcr} all hold with $|H/(H \cap N^{\circ})|$ in place of $|\pi(H)/\pi(H \cap N^{\circ})|$ and other similar adjustments. While this makes for slightly cleaner statements, it also misses some pathological cases, such as letting $K$ be a connected reductive group and $G = H = K \times S$ for a finite $p$-group $S$, where $p = \Char(k) > 0$. In this case, the question of whether $H$ is separable for $K$ does not depend on properties of $S$, although $H/(H \cap K) \cong S$ so $\Char(k)$ divides $|H/(H \cap G^{\circ})|$ if $|S| > 0$.
\end{remark}

\section{Relative complete reducibility of Lie subalgebras} \label{sec:liealg}

In this section we consider the analogue of Theorem~\ref{THM:MAIN} for subalgebras of $\Lie(G)$. We write $\mathfrak{g} = \Lie(G)$, and we denote the Lie algebras of the subgroups $K$, $N$, $C$, $P_{\lambda}$, $L_{\lambda}$ of $G$ by $\mathfrak{k}$, $\mathfrak{n}$, $\mathfrak{c}$, $\mathfrak{p}_{\lambda}$ and $\mathfrak{l}_{\lambda}$, respectively.

Recall from \cite[Definition 3.9]{Bate2011} that a subalgebra $\mathfrak{h}$ of $\mathfrak{g}$ is called relatively $G$-completely reducible with respect to a reductive subgroup $K \le G$ if, whenever $\mathfrak{h} \subseteq \mathfrak{p}_{\lambda}$ for $\lambda \in Y(K)$, there exists $\mu \in Y(K)$ such that $P_{\lambda} = P_{\mu}$ and $\mathfrak{h} \subseteq \mathfrak{l}_{\mu}$. If this holds for $K = G$ then $\mathfrak{h}$ is called $G$-completely reducible, cf.~\cite{McNinch2007}.

\begin{lemma} \label{lem:liealg-parabolic}
In the notation of Theorem~\ref{THM:MAIN}, for all $\lambda \in Y(K)$ the following hold.
\begin{enumerate}
	\item $\pi(P_{\lambda}^{\circ} \cap N) = P_{\pi \circ \lambda}^{\circ}$, $\pi(L_{\lambda}^{\circ} \cap N) = L_{\pi \circ \lambda}^{\circ}$. \label{lie-parab-i}
	\item $\pi^{-1}(P_{\pi \circ \lambda}^{\circ}) = P_{\lambda}^{\circ} \cap N$, $\pi^{-1}(L_{\pi \circ \lambda}^{\circ}) = L_{\lambda}^{\circ} \cap N$. \label{lie-parab-ii}
	\item $\Lie(P_{\lambda} \cap N) = \mathfrak{p}_{\lambda} \cap \mathfrak{n}$, $\Lie(L_{\lambda} \cap N)  = \mathfrak{l}_{\lambda} \cap \mathfrak{n}$. \label{lie-parab-iii}
	\item $d\pi(\mathfrak{p}_{\lambda} \cap \mathfrak{n}) = \mathfrak{p}_{\pi \circ \lambda}$, $d\pi(\mathfrak{l}_{\lambda} \cap \mathfrak{n}) = \mathfrak{l}_{\pi \circ \lambda}$. \label{lie-parab-iv}
	\item $(d\pi)^{-1}(\mathfrak{p}_{\pi \circ \lambda}) = \mathfrak{p}_{\lambda} \cap \mathfrak{n}$, $(d\pi)^{-1}(\mathfrak{l}_{\pi \circ \lambda}) = \mathfrak{l}_{\lambda} \cap \mathfrak{n}$. \label{lie-parab-v}
\end{enumerate}
\end{lemma}

\begin{proof} For all $\lambda \in Y(K)$ we have $P_{\pi \circ \lambda}^{\circ} = P_{\pi \circ \lambda} \cap (N/C)^{\circ}$, since this is a parabolic subgroup of the connected reductive subgroup $(N/C)^{\circ}$, and similarly $L_{\pi \circ \lambda}^{\circ} = L_{\pi \circ \lambda} \cap (N/C)^{\circ}$. Now recall the subgroup $M$ from Lemma~\ref{lem:usefulsub}, and that the restriction $\pi \colon M \to N/C$ is an isogeny. Thus $\pi \colon M^{\circ} \to (N/C)^{\circ}$ is a surjective map of connected reductive groups. Thus part \ref{lie-parab-i} follows from \cite[Lemma 2.11(i)]{Bate2005}. For part \ref{lie-parab-ii}, using part \ref{lie-parab-i} we see that $\pi^{-1}(P_{\pi \circ \lambda}^{\circ}) = \pi^{-1}(\pi(P_{\lambda}^{\circ})) = (P_{\lambda}^{\circ} \cap N)C = P_{\lambda}^{\circ} \cap N$, and similarly for $L_{\lambda}$.

Part \ref{lie-parab-iii} follows from the proof of \cite[Theorem 13.4.2(ii)]{Springer2009}. The proof shows in particular that the containment (63) given on p.\ 234 there is an equality; this is the desired result for $P_{\lambda}$. Also $\Lie(L_{\lambda} \cap N)$ is the subalgebra of $\mathfrak{n}$ centralised by $\lambda(k^{\ast})$, which is precisely $\mathfrak{l}_{\lambda} \cap \mathfrak{n}$.

For part \ref{lie-parab-iv}, the left-hand side is clearly contained in the right-hand side. On the other hand $\mathfrak{n} = \Lie(N^{\circ})$ and the restriction $\pi \colon N^{\circ} \to N^{\circ}/(N^{\circ} \cap C)$ is a quotient of $N^{\circ}$ by the closed subgroup $N^{\circ} \cap C$ and is therefore separable \cite[5.5.6(ii)]{Springer2009} (cf.\ also \cite[Exercise 5.5.9(5)(c)]{Springer2009}), so $d\pi$ induces an isomorphism $\mathfrak{n}/\mathfrak{c} \to \Lie(N/C)$. Therefore,
\begin{align*}
\dim d\pi( \mathfrak{p}_{\lambda} \cap \mathfrak{n}) + \dim \mathfrak{c} &= \dim (\mathfrak{p}_{\lambda} \cap \mathfrak{n}) \\
&= \dim (P_{\lambda}^{\circ} \cap N) \\
&= \dim \pi(P_{\lambda}^{\circ} \cap N) + \dim C \\
&= \dim P_{\pi \circ \lambda} + \dim C \\
&= \dim \mathfrak{p}_{\pi \circ \lambda} + \dim C,
\end{align*}
where the first equality uses the separability of the restriction of $\pi$ to $N^{\circ}$ and the second equality uses \ref{lie-parab-iii}. We thus deduce that $\dim d\pi( \mathfrak{p}_{\lambda} \cap \mathfrak{n}) = \dim \mathfrak{p}_{\pi \circ \lambda}$, and so the two subalgebras are equal. An identical argument shows that $d\pi(\mathfrak{l}_{\lambda} \cap \mathfrak{n}) = \mathfrak{l}_{\pi \circ \lambda}$.

Part \ref{lie-parab-v} follows directly from part \ref{lie-parab-iv}, as $(d\pi)^{-1}(\mathfrak{p}_{\pi\circ\lambda})=(d\pi)^{-1}(d\pi(\mathfrak{p}_\lambda \cap \mathfrak{n} ))=\mathfrak{p}_\lambda \cap \mathfrak{n}$.
Similarly for $\mathfrak{l}_\lambda$.
\end{proof}

Our analogue of Theorem~\ref{THM:MAIN} for Lie algebras is now as follows.

\begin{theorem} \label{THM:LIE}
Let $K \le G$ be reductive algebraic groups. 
 Let $N = N_{G}(K^{\circ}) \le N_{G}(\mathfrak{k})$ and $C = C_{G}(K^{\circ}) \le C_{G}(\mathfrak{k})$, let $\pi \colon N \to N/C$ be the quotient map, and write $d\pi$ for the differential $\mathfrak{n} \to \mathfrak{n}/\mathfrak{c} = \Lie(N/C)$.

Let $\mathfrak{h}$ be a Lie subalgebra of $\mathfrak{n}$. Then $\mathfrak{h}$ is relatively $G$-completely reducible with respect to $K$ if and only if $d\pi(\mathfrak{h})$ is $\pi(N)$-completely reducible.
\end{theorem}

\begin{proof} The map $\pi \colon N\to N/C$ induces a surjection $Y(K)\to Y(N/C)$.
It follows from Lemma~\ref{lem:liealg-parabolic}(iv), (v) 
that $\mathfrak{h}$ is contained in $\mathfrak{p}_\lambda$ respectively $\mathfrak{l}_\lambda$ for $\lambda \in Y(K)$ if and only if $d\pi(\mathfrak{h})$ is contained in $\mathfrak{p}_{\pi\circ\lambda}$ (respectively $\mathfrak{l}_{\pi \circ \lambda}$). 
\end{proof}

The following is the Lie algebra counterpart of Corollary~\ref{cor:NORMAL-II}.
Since both $\mathfrak{n}/\mathfrak{c}$ and the trivial subalgebra in $\mathfrak{n}/\mathfrak{c}$ are $\pi(N)$-cr, 
Theorem~\ref{THM:LIE} gives the following.

\begin{corollary} \label{cor:NORMAL-II-LIE}
Let $K \le G$ be reductive algebraic groups. Then the subalgebras $\mathfrak{n}_{\mathfrak{g}}(K^{\circ})$ and 
$\mathfrak{c}_{\mathfrak{g}}(K^{\circ})$ of $\mathfrak{g}$ 
are relatively $G$-completely reducible with respect to $K$.
\end{corollary}

As with relative $G$-complete reducibility of subgroups, a Lie subalgebra $\mathfrak{h}$ of $\mathfrak{g}$ is relatively $G$-cr with respect to a reductive subgroup $K \le G$ precisely when the $K$-orbit of any (equivalently, of every) finite tuple $\mathbf{h} \in \mathfrak{h}^{n}$ which generates $\mathfrak{h}$ as a Lie algebra is closed in $\mathfrak{g}^{n}$ \cite[Theorem 3.10(iii)]{Bate2011}. Thus results from geometric invariant theory can be brought to bear. In particular, if $\mathfrak{h}$ is not relatively $G$-cr with respect to $K$, then there exists a so-called `optimal destabilising parabolic subgroup' for $\mathfrak{h}$, see \cite[Definition 3.23 and Remark 3.24]{Bate2011}. This is a \emph{canonical} parabolic subgroup $P_{\lambda}$ $(\lambda \in Y(K))$ such that $\mathfrak{h} \subseteq \mathfrak{p}_{\lambda}$ and $\mathfrak{h} \nsubseteq \mathfrak{l}_{\lambda}$.

The following result generalises \cite[Theorem 1(2)]{McNinch2007}. It follows directly from this absolute result and Theorem~\ref{THM:LIE}, but can also be proved directly by considering optimal destabilising parabolic subgroups, mirroring \cite[Example 5.29]{Bate2013}.

\begin{theorem}
\label{THM:relative-McNinch}
In the notation of Theorem~\ref{THM:MAIN}, if $H$ is relatively $G$-completely reducible with respect to $K$, then the Lie algebra $\mathfrak{h}$ of $H$ is also relatively $G$-completely reducible with respect to $K$.
\end{theorem}

\begin{proof}
Suppose that $\mathfrak{h}$ is not relatively $G$-cr with respect to $K$ and let $P_{\lambda}$ ($\lambda \in Y(K)$) be the optimal destabilising parabolic subgroup for $\mathfrak{h}$. By the optimality of $P_{\lambda}$ (\cite[Remark 3.24]{Bate2011}), we have $N_{N_{G}(\mathfrak{k})}(\mathfrak{h}) \le P_{\lambda}$, and since $H \le N_{G}(\mathfrak{h})$ and $H \le N \le N_{G}(\mathfrak{k})$ we have $H \le P_{\lambda}$ also. Then by hypothesis we have $H \le L_{\mu}$ for some $\mu \in Y(K)$ such that $P_{\lambda} = P_{\mu}$, and therefore $\mathfrak{h} \le \mathfrak{l}_{\mu}$.
\end{proof}

Note that the converse of Theorem~\ref{THM:relative-McNinch} already fails in the absolute case \cite{McNinch2007}.

Here is the counterpart of Theorem~\ref{THM:GEOMETRIC} for Lie algebras, which is equivalent to 
Theorem~\ref{THM:LIE}, thanks to \cite[Theorem 3.10(iii)]{Bate2011} which is the analogue of Theorem~\ref{thm:gcr_gen_tup}
for Lie subalgebras of $ \mathfrak{g}$ and their generating tuples.
 
\begin{theorem} 
\label{THM:GEOMETRIC-LIE}
Let $K \le G$ be reductive algebraic groups, write $N = N_{G}(K^{\circ})$, $C = C_{G}(K^{\circ})$, and let $\pi \colon N \to N/C$ be the quotient map,
and write $d\pi$ for the differential $\mathfrak{n} \to \mathfrak{n}/\mathfrak{c} = \Lie(N/C)$. 
Let $\mathbf{h} \in \mathfrak{n}^{n}$ $(n \ge 1)$ and write $d\pi$ also for the map $\mathfrak{n}^{n} \to (\mathfrak{n}/\mathfrak{c})^{n}$

Then $K \cdot \mathbf{h}$ is closed in $\mathfrak{g}^{n}$ if and only if $\pi(N) \cdot d\pi(\mathbf{h})$ is closed in $(\mathfrak{n}/\mathfrak{c})^{n}$.
\end{theorem}

The following is the analogue of Theorem~\ref{THM:geom} for the diagonal action of $K$ on $\mathfrak{g}^{n}$.

\begin{corollary}
\label{cor:geom-LIE}
With the above notation, the following are equivalent.
\begin{enumerate}
	\item $K \cdot \mathbf{h}$ is closed in $\mathfrak{g}^{n}$. \label{lie-geom-i}
	\item $\pi(K) \cdot d\pi(\mathbf{h})$ is closed in  $(\mathfrak{n}/\mathfrak{c})^{n}$. \label{lie-geom-ii}
	\item $(d\pi)^{-1}(d\pi(K \cdot \mathbf{h}))$ is closed in $\mathfrak{g}^{n}$. \label{lie-geom-iii}
	\item Every orbit of $K$ on $(d\pi)^{-1}(d\pi(K \cdot \mathbf{h}))$ is closed in $\mathfrak{g}^{n}$. \label{lie-geom-iv}
\end{enumerate}
\end{corollary}

\begin{proof} 
The equivalence of \ref{lie-geom-i} and \ref{lie-geom-ii} is given by Theorem~\ref{THM:GEOMETRIC-LIE}. The equivalence of \ref{lie-geom-ii} and \ref{lie-geom-iii} follows from the fact that the topology on $(\mathfrak{n}/\mathfrak{c})^{n}$ is the quotient topology, and we also use the fact that $d\pi$ is surjective (as shown in the proof of Lemma~\ref{lem:liealg-parabolic}\ref{lie-parab-iv} above).

Let $K = \cup_{i=1}^t K^\circ x_i$ be the decomposition of $K$ into right cosets of $K^\circ$. Writing $\mathbf{h} = (h_{1},\ldots,h_{n})$ we have 
\begin{align*}
(d\pi)^{-1}(d\pi(K \cdot \mathbf{h})) 
&= \bigcup_{i=1}^t (d\pi)^{-1}(d\pi \left(K^\circ \cdot (x_i \cdot \mathbf{h}) \right)\\
&= \bigcup_{i=1}^t  \{ (x \cdot (x_i \cdot h_1)+c_1,\ldots, x \cdot ( x_i \cdot h_n))+c_n) \ssep x \in K^\circ,\ c_{i} \in  \mathfrak{c}\},\\
&= \bigcup_{i=1}^t  \{ (x \cdot ((x_i  \cdot h_1)+c_1),\ldots, x \cdot ((x_i \cdot h_n) +c_n)) \ssep x \in K^\circ,\ c_{i} \in  \mathfrak{c}\},
\end{align*}
where the last equality follows since $K^{\circ}$ centralises $\mathfrak{c}$.

Now if $K \cdot \mathbf{h}$ is closed in $\mathfrak{g}^{n}$, then so is $K^\circ \cdot (x_i \cdot \mathbf{h})$ for each $i$.
Thus $(d\pi)^{-1}(d\pi(K\cdot \mathbf{h}))$ is a union of $K^\circ$-orbits, each of which is translated to one of the closed $K^\circ$-orbits 
$K^\circ \cdot (x_i \cdot \mathbf{h})$ by
an element of $\mathfrak{c}^{n}$.
As translation is a variety automorphism $\mathfrak{g}^{n} \to \mathfrak{g}^{n}$, it follows that every $K^\circ$-orbit in $(d\pi)^{-1}(d\pi(K \cdot \mathbf{h}))$ is closed in $\mathfrak{g}^{n}$.
Consequently, every $K$-orbit in $(d\pi)^{-1}(d\pi(K \cdot \mathbf{h}))$ is closed in $\mathfrak{g}^{n}$ as well.
So \ref{lie-geom-iv} follows from \ref{lie-geom-i} and the reverse implication is clear. 
\end{proof}

Comments concerning the equivalences in Corollary~\ref{cor:geom-LIE} similar to those in Remark~\ref{rem:geometric} apply. Moreover, one defines relative $G$-irreducibility of Lie subalgebras of $\mathfrak{g}$ (with respect to $K$) in terms of $K$-stable points on $\mathfrak{g}^{n}$ \cite[Definition 3.12]{Bate2011}, and results analogous to Corollary~\ref{cor:irred} and \ref{cor:irredgeom} apply in this case, with the obvious modifications.

\section{Changing the field} \label{sec:rationality}

In this section we prove Theorem~\ref{THM:MAIN-RATIONAL}, which generalises Theorem~\ref{THM:MAIN} to arbitrary fields. This allows us to generalise many other results, including Theorem~\ref{THM:GEOMETRIC}. First we recall some relevant notions from \cite[\S 4]{Bate2011}.
In this section $k$ denotes an arbitrary field and $\overline k$ is the algebraic closure of $k$. Algebraic groups and varieties are taken to be defined
over $\overline k$. If $V$ is a $k$-variety and $k'/k$ is an algebraic extension then we denote the set of $k'$-points of $V$ by $V(k')$. The set of $k$-defined cocharacters of a $k$-group $M$ is denoted $Y_k(M)$.  
We say that a $G$-variety $V$ is defined over $k$ if both $V$ and the action of $G$ on $V$ are defined over $k$.

We begin with the definition of relative $G$-complete reducibility over $k$, 
\cite[Definition 4.1]{Bate2011}.

\begin{definition}
\label{defn:relativelyk}
Let $K \le G$ be reductive algebraic $k$-groups.
A subgroup $H$ of $G$
is \emph{relatively $G$-completely reducible over $k$ with respect to $K$} if
for every $\lambda \in Y(K)$ such that $P_\lambda$ is $k$-defined and 
$H\subseteq P_\lambda$,
there exists $\mu \in Y(K)$ such that
$P_\lambda = P_\mu$, $L_\mu$ is $k$-defined and $H \subseteq L_\mu$.
If $K=G$, then we say
that $H$ is
\emph{$G$-completely reducible over $k$}, see also 
\cite{Serre2003-2004} and \cite[\S 5]{Bate2005}.
\end{definition}

When $k = \overline{k}$ this definition coincides with Definition~\ref{def:gcr}, since in this case each R-parabolic subgroup and R-Levi subgroup of $G$ is $k$-defined.

According to \cite[Lemma~4.8]{Bate2011}, when discussing relative $G$-complete reducibility over $k$ with respect to $K$,
it suffices to consider R-parabolic subgroups and R-Levi subgroups of the form $P_{\lambda}$ and $L_{\lambda}$ for 
$\lambda \in Y_k(K)$,
rather than all $k$-defined R-parabolic subgroups and R-Levi subgroups arising from a cocharacter of $K$.

\subsection{Rational analogue of Theorem~\ref{THM:MAIN}}

We now embark on proving the rational version of Theorem~\ref{THM:MAIN}. We need to generalise results from Section~\ref{sec:proof} and from \cite{Bate2005}. From now on we suppose that in the setting of Theorem~\ref{THM:MAIN}, the groups $G$, $K$, $N$ and $C$ are all $k$-defined. Then the quotient map $\pi \colon N \to N/C$ is a $k$-defined morphism. Let $M$ be a reductive subgroup of $G$ guaranteed by Lemma~\ref{lem:usefulsub}, and let $\pi_{M}$ denote the restriction of $\pi$ to $M$, so that $\pi_{M}$ is an isogeny $M \to N/C$.

\begin{lemma} \label{lem:usefulsub-rational}
With the above assumptions, the subgroup $M$ may be taken to be $k$-defined, and then $\pi_{M}$ is a $k$-defined isogeny.
\end{lemma}

\begin{proof} By assumption, $K$ is $k$-defined, hence so are $K^{\circ}$ and $\Kder$, \cite[Ch.~I 1.2 Proposition (b); 2.3 Corollary]{Borel1991}. Therefore so is the quotient map $N \to N/\Kder$. Now \cite[Theorem 1.1]{Brion2015} applies to the more general setting of $k$-group schemes with finite quotients, hence the subgroup $M$ constructed in Lemma~\ref{lem:usefulsub} is the pre-image (under $\pi$) of a finite $k$-defined subgroup of $N/C$, hence is a $k$-defined subgroup of $G$. Now $\pi$ is $k$-defined hence so is its restriction to $M$.
\end{proof}

\begin{lemma} \label{lem:k-cochar-morphism}
Keeping the above assumptions, suppose that $C$ is normal in $G$. If $\lambda \in Y_{k}(G/C)$ then there exists $\mu \in Y_{k}(G/C)$ such that $P_{\mu} = P_{\lambda}$, $L_{\mu} = L_{\lambda}$ and $\mu = \pi \circ \nu$ for some $\nu \in Y_{k}(M)$.
\end{lemma}

\begin{proof} 
	With $M$ and $\pi_{M}$ as above, $\pi_M$ is an isogeny and hence $\pi_M$ is quasi-central. Since $d\pi:\mathfrak{n}\to \mathfrak{n/c}$ is the quotient map, the kernel of $d\pi_M$ is just $\Lie(M)\cap \Lie(C_G(K^\circ))\subseteq \Lie(M)\cap \Lie(C_G(M^\circ))$, and the latter is central in $\Lie(M)$. Then $\pi_M$ is central, cf.~\cite[\S 22.3]{Borel1991}. Thus, by \cite[22.5 Corollary]{Borel1991}, the preimage of a $k$-defined torus is $k$-defined. Let $T = \lambda(\overline{k}^{\ast})$ and let $S = \pi_{M}^{-1}(T)^{\circ}$. Thus $S$ and $T$ are $1$-dimensional $k$-defined tori. Since $\pi$ is a $k$-morphism, it induces a map $Y_{k}(S) \to Y_{k}(T)$ whose image has finite index, say $n$. Then $\mu = n\lambda$ satisfies the required conditions.
\end{proof}

The following is the rational version of Theorem~\ref{THM:MAIN} in the special case that $G = N$ and $C^{\circ}$ is a torus. The proof of Theorem~\ref{THM:MAIN-RATIONAL} below proceeds by reducing to this special case.
\begin{lemma} \label{lem:mainthm-rational-nondegenerate}
Let $G$ be a $k$-defined reductive group. Let $C$ be a $k$-defined normal subgroup of $G$ such that $C^{\circ}$ is a torus, and let $\pi$ be the quotient map $G \to G/C$. Let $M$ be a subgroup guaranteed by Lemma~\ref{lem:usefulsub-rational}, so that $M \cap C$ is finite and $\pi$ induces an isogeny $\pi_{M} \colon M \to G/C$.

Then a subgroup $H$ of $G$ is $G$-completely reducible over $k$ if and only if $\pi(H)$ is $(G/C)$-completely reducible over $k$.
\end{lemma}

\begin{proof}
As in the proof of Theorem~\ref{THM:MAIN} it suffices to assume that $H \le M$, since $H$ is contained in a parabolic subgroup $P_{\lambda}$ ($\lambda \in Y_{k}(G)$) if and only if $HC \le P_{\lambda}$ if and only if $HC \cap M \le P_{\lambda}$, and similarly for $L_{\lambda}$.

So suppose $H$ is $G$-cr over $k$, and suppose that $\pi(H) \le P_{\mu}(G/C)$ where $\mu \in Y_{k}(G/C)$. By Lemma~\ref{lem:k-cochar-morphism} we can assume that $\mu = \pi \circ \lambda$ for some $\lambda \in Y_{k}(M)$. Then $H \le \pi^{-1}(\pi(H)) \le \pi^{-1}(P_{\mu}) = P_{\lambda}(G)$, by Proposition~\ref{prop:usefulsub}. Since $H$ is $G$-cr over $k$, by \cite[Lemma 2.5(iii)]{Bate2013} there exists $u \in R_{u}(P_{\lambda})(k)$ such that $H \le L_{u \cdot \lambda}$. Now by \cite[Lemma 6.15(iii)]{Bate2005} we have $u \in R_{u}(P_{\lambda} \cap M)(k)$, so $u \cdot \lambda \in Y_{k}(M)$. Thus using Proposition~\ref{prop:usefulsub} again, we have $\pi(H) \le \pi(L_{u \cdot \lambda} \cap M) = L_{\pi(u) \cdot \mu}$. Moreover $P_{\pi(u) \cdot \mu} = \pi(P_{u \cdot \lambda}) = \pi(P_{\lambda}) = P_{\mu}$, also by Proposition~\ref{prop:usefulsub}, so $L_{\pi(u) \cdot \mu}$ is an R-Levi subgroup of $P_{\mu}$. It follows that $\pi(H)$ is $(G/C)$-cr over $k$.

Conversely, suppose that $\pi(H)$ is $(G/C)$-cr over $k$, and suppose that $H \le P_{\lambda}$ for some $\lambda \in Y_{k}(G)$. Then $\pi(H) \le P_{\mu}$ where $\mu = \pi \circ \lambda \in Y_{k}(G/C)$. As $\pi(H)$ is $(G/C)$-cr over $k$, there is an R-Levi subgroup $L$ of $P_{\mu}$ such that $\pi(H) \le L$. By the proof of Lemma~\ref{lem:k-cochar-morphism}, we can replace $\mu$ and $\lambda$ by some positive integer multiples (without changing the corresponding R-parabolic or R-Levi subgroups) so that there exists $\sigma \in Y_{k}(M)$ with $\pi \circ \sigma = \mu$. We have $\lambda(\overline{k}^{\ast}) \le \sigma(\overline{k}^{\ast})C^{\circ}$, and it follows that there exists a $k$-cocharacter $\tau \in Y_{k}(C)$ such that $\lambda = \sigma + \tau$.

Since $\pi(H)$ is $(G/C)$-cr over $k$, there exists $\nu \in Y_{k}(G/C)$ so that $P_{\nu} = P_{\mu}$ and $\pi(H) \le L_{\nu}$. By Lemma~\ref{lem:k-cochar-morphism} we can choose $\nu$ so that $\nu = \pi \circ \sigma'$ for some $\sigma' \in Y_{k}(M)$. Using Proposition~\ref{prop:usefulsub} we have $P_{\sigma'} = \pi^{-1}(P_{\nu}) = \pi^{-1}(P_{\mu}) = P_{\sigma}$. By \cite[Corollary 2.6]{Bate2013}, we can adjust $\nu$ and $\sigma'$ (without affecting the corresponding R-parabolic and R-Levi subgroups) so that $\sigma' = u \cdot \sigma$ for some $u \in R_{u}(P_{\sigma} \cap M)(k) = R_{u}(P_{\sigma})(k) = R_{u}(P_{\lambda})(k)$. Thus, replacing $\lambda$ by $u \cdot \lambda$ if necessary, we can assume that $\pi(H) \le L_{\mu}$. We have $H \le \pi_{M}^{-1}(L_{\mu}) = L_{\sigma} \cap M$ by Proposition~\ref{prop:usefulsub}. We have $\lambda \in Y_{k}(L_{\sigma})$, and
\[ H \le P_{\lambda} \cap (L_{\sigma} \cap M) \le P_{\lambda} \cap L_{\sigma} = P_{\tau} \cap L_{\sigma}. \]
Now $L_{\sigma}^{\circ} \le (L_{\sigma} \cap M^{\circ})C^{\circ} \le P_{\tau} \cap L_{\sigma}$, hence $R_{u}(P_{\tau} \cap L_{\sigma}) = \{1\}$, and so $P_{\tau} \cap L_{\sigma} = L_{\tau} \cap L_{\sigma}$. Thus $H \le L_{\tau} \cap L_{\sigma} \le L_{\lambda}$, and so $H$ is $G$-cr over $k$.
\end{proof}

The following is now the rational version of Theorem~\ref{THM:MAIN}.

\begin{theorem} \label{THM:MAIN-RATIONAL}
Let $K \le G$ be reductive algebraic $k$-groups, write $N = N_{G}(K^{\circ})$, $C = C_{G}(K^{\circ})$, and let $\pi \colon N \to N/C$ be the quotient map. Suppose that $N$ and $C$ are $k$-defined, and let $H \le N$. 

Then $H$ is relatively $G$-completely reducible over $k$ with respect to $K$ if and only if $\pi(H)$ is $\pi(N)$-completely reducible  over $k$.
\end{theorem}

\begin{proof} The proof of Theorem~\ref{THM:MAIN} begins by reducing to the case that $G = N = MZ(K^{\circ})^{\circ}$, where $M$ is a reductive subgroup guaranteed by Lemma~\ref{lem:usefulsub}. Using Lemma~\ref{lem:usefulsub-rational}, the same reduction holds in this rational setting, since none of the arguments require special properties of the field $k$, only that the groups and quotient maps involved are defined over $k$. So under this reduction, $Y_k(G) = Y_k(G^{\circ}) = Y_k(K^{\circ}) = Y_k(K)$, so that $H$ is relatively $G$-cr over $k$ with respect to $K$ if and only if $H$ is $G$-cr over $k$. Moreover after reduction to this case, the group $C^{\circ} = Z(K^{\circ})^{\circ}$ is a torus, and so the required conclusion follows from Lemma~\ref{lem:mainthm-rational-nondegenerate}.
\end{proof}

\begin{remark}
Note that while $G$, $K$, $N$ and $C$ all need to be $k$-defined for the above proofs to work, we do not make any such assumption on $H$.
\end{remark}

\subsection{Rational analogues of Theorem~\ref{THM:GEOMETRIC} and corollaries}
In order to generalise Theorem~\ref{THM:GEOMETRIC} to arbitrary fields,
we need a notion of a ``closed orbit'' for a group $M(k)$ of $k$-points
of a reductive $k$-group $M$ acting on a $k$-variety.
The correct notion for us is as follows, \cite[Definition~3.8]{Bate2013}:

\begin{definition}\label{defn:cocharclosed}
Let $M$ be a reductive $k$-group and let $V$ be an $M$-variety defined over $k$.
Let $v \in V$. We say that the $M(k)$-orbit $M(k)\cdot v$ is
\emph{cocharacter-closed over $k$} if for
any $\lambda \in Y_k(M)$ such that
$v' = \lim_{a\to 0}\lambda(a)\cdot v$ exists, $v'$ is $M(k)$-conjugate
to $v$.
\end{definition}

The usefulness of this notion is shown by the following 
characterisation of relative 
$G$-complete reducibility over $k$ in terms of 
cocharacter-closure of a rational orbit of a generic tuple.
Combining 
\cite[Corollary~5.3]{Bate2017}, \cite[Theorem~4.12 (iii)]{Bate2011}, and 
\cite[Theorem 9.3]{Bate2017},
we obtain the following rational version of Theorem~\ref{thm:gcr_gen_tup}. For $K = G$ this is just \cite[Theorem 9.3]{Bate2017}.

\begin{theorem}
\label{thm:crvscocharclosed}
Let $K \le G$ be reductive algebraic $k$-groups.
Let $H$ be a subgroup of $G$ and let $\mathbf{h}\in H^n$ 
be a generic tuple of $H$.
Then $H$ is relatively $G$-completely reducible over $k$ with respect to $K$ 
if and only if
$K(k) \cdot \mathbf{h}$ is cocharacter-closed over $k$.
\end{theorem}

Owing to Theorem~\ref{thm:crvscocharclosed},   
Theorem~\ref{THM:MAIN-RATIONAL} is equivalent to the following rational version of 
Theorem~\ref{THM:GEOMETRIC}.

\begin{theorem} 
\label{THM:GEOMETRIC-RATIONAL}
Let $K \le G$ be reductive algebraic $k$-groups, write $N = N_{G}(K^{\circ})$, $C = C_{G}(K^{\circ})$, and let $\pi \colon N \to N/C$ be the quotient map. Suppose that $N$ and $C$ are $k$-defined, let $\mathbf{h} \in N^{n}$ $(n \ge 1)$ and write $\pi$ also for the map $N^{n} \to (N/C)^{n}$.

Then $K(k) \cdot \mathbf{h}$ is cocharacter-closed over $k$ if and only if $\pi(N)(k) \cdot \pi(\mathbf{h})$ is cocharacter-closed over $k$.
\end{theorem}

Next, Theorem~\ref{THM:MAIN-RATIONAL} allows us to immediately deduce the relative version of \cite[Corollary~9.7]{Bate2017}.
However the result holds without the restriction $H \leq N$.
One implication of the first part holds without the condition on $C_K(H)$, by \cite[Theorem~4.13]{Bate2011}.
If $k'/k$ is an algebraic extension of perfect fields then both implications 
of the first part hold without the condition on $C_K(H)$, by \cite[Theorem~4.14]{Bate2017}.
The next corollary therefore follows at once from \cite[Theorem~5.7]{Bate2017} and Theorem~\ref{thm:crvscocharclosed}.

\begin{corollary}
Let $K \le G$ be reductive algebraic $k$-groups
	and let $H$ be a $k$-defined subgroup of $G$ such that $C_K(H)$ is $k$-defined.
	Then the following hold.
	\begin{enumerate}
		\item For any separable algebraic extension $k'/k$, $H$ is relatively 
		$G$-completely reducible over $k'$ with respect to $K$ if and only if $H$
		is relatively $G$-completely reducible over $k$ with respect to $K$.
		\item For a $k$-defined torus $S$ of $C_K(H)$ let $L=C_K(S)$.
		Then $H$ is relatively $G$-completely reducible over $k$ with respect to $K$
		if and only if $H$ is relatively $G$-completely reducible over $k$ with
		respect to $L$.
	\end{enumerate}
\end{corollary}

We end this section with a geometric consequence of Theorem~\ref{THM:GEOMETRIC-RATIONAL} which 
provides a rational version of Theorem~\ref{THM:geom}.
For that we require the notion of a cocharacter closed subset of a $G$-variety over $k$, \cite[Definition~1.2(a)]{Bate2017}.

\begin{definition}
\label{defn:cocharclosure}
Let $V$ be an affine variety over $k$ on which $G$ acts. 
Given a subset $X$ of $V$, we say that $X$ 
is \emph{cocharacter-closed (over $k$)} if for every $v \in X$ and
every $\lambda \in Y_k(G)$ such that $v' = \lim_{a \to 0} \lambda(a) \cdot v$ 
exists, $v' \in X$.
\end{definition}

Note that this definition coincides with Definition~\ref{defn:cocharclosed} 
 if $X = G(k) \cdot v$ for some $v \in V$.

The cocharacter-closed subsets of $V$ form the closed sets of a topology on $V$
(it is clear that arbitrary intersections and unions of cocharacter-closed sets are cocharacter-closed,
and that the empty set and the whole space $V$ are cocharacter-closed), cf.~\cite[Remark~3.1(iii)]{Bate2017}.
It is this topology which is used in our rational version of Theorem~\ref{THM:geom}.

\begin{theorem}
\label{THM:rational-geom}
Let $K \le G$ be reductive algebraic $k$-groups, write $N = N_{G}(K^{\circ})$, $C = C_{G}(K^{\circ})$, and let $\pi \colon N \to N/C$ be the quotient map. Let $\mathbf{h} \in N^{n}$ $(n \ge 1)$ and write $\pi$ also for the map $N^{n} \to (N/C)^{n}$. Suppose that $N$ and $C$ are $k$-defined.
Then the following are equivalent.
\begin{enumerate}
	\item $K(k) \cdot \mathbf{h}$ is cocharacter-closed over $k$  in $G^{n}$. \label{rat-geom-i}
	\item $\pi(K(k)) \cdot \pi(\mathbf{h})$ is cocharacter-closed over $k$ in $(N/C)^{n}$. \label{rat-geom-ii}
	\item $\pi^{-1}(\pi(K(k) \cdot \mathbf{h}))$ is cocharacter-closed  over $k$ in $G^{n}$. \label{rat-geom-iii}
	\item Every $K(k)$-orbit on $\pi^{-1}(\pi(K(k) \cdot \mathbf{h}))$ is cocharacter-closed over $k$ in $G^{n}$. \label{rat-geom-iv}
\end{enumerate}
\end{theorem}

\begin{proof} Thanks to Theorem~\ref{THM:GEOMETRIC-RATIONAL}, \ref{rat-geom-i} and \ref{rat-geom-ii} are equivalent. Since $N^{n}$ is closed in $G^{n}$ 
and stabilised by $K$, it is cocharacter-closed for the action of $K$ by \cite[Remark~3.1(ii)]{Bate2017}.
Further, since $\pi(K(k))\cdot \pi(\mathbf{h}) = \pi(K(k) \cdot \mathbf{h})$ and the topology on $(N/C)^{n}$ is the quotient topology afforded by the cocharacter-closed subsets of $N$, it follows that \ref{rat-geom-ii} and \ref{rat-geom-iii} are equivalent. 

Clearly, \ref{rat-geom-iv} implies \ref{rat-geom-i}. Now suppose \ref{rat-geom-i}.
Thanks to \cite[Corollary~3.5]{Bate2017},
$K(k) \cdot \mathbf{h}$ is cocharacter-closed over $k$ if and only if $K^\circ(k) \cdot \mathbf{h}$ is cocharacter-closed over $k$. Thus we may assume that $K$ is connected.
 Writing $\mathbf{h} = (h_1,\ldots,h_n)$ we have
\begin{align*}
\pi^{-1}(\pi(K(k) \cdot \mathbf{h})) 
&= \left\{ \left((x \cdot h_1)c_1,\ldots, (x \cdot h_n)c_n\right) \ssep x \in K(k),\ c_{i} \in C \right\},\\
&= \left\{ \left(x \cdot (h_1c_1),\ldots, x \cdot (h_nc_n)\right) \ssep x \in K(k),\ c_{i} \in C \right\},
\end{align*}
where the last equality follows again since
$K = K^\circ$ centralises $C$.
Thus $\pi^{-1}(\pi(K(k) \cdot \mathbf{h}))$ is a union of $K(k)$-orbits, 
each of which is a translate of the cocharacter-closed $K(k)$-orbit  
$K(k) \cdot \mathbf{h}$ by
an element of $C^{n}$.
As translation is a $k$-variety automorphism $G^{n} \to G^{n}$, 
it follows that every $K(k)$-orbit on $\pi^{-1}(\pi(K(k) \cdot \mathbf{h}))$ is cocharacter-closed in $G^{n}$,
as desired.
\end{proof}

\bigskip

Acknowledgments: 
The second author acknowledges support from the Alexander von Humboldt Foundation.
The authors wish to thank Ben Martin for helpful comments on early versions of the paper.

\end{document}